\documentclass[12pt,reqno]{article}
mm \textheight=8.9in

\usepackage{amssymb}
\usepackage{amsmath}
\usepackage{amsthm}
\usepackage{pb-diagram}
\usepackage{color}
\newcommand{\A}{{\mathcal{A}}}

\newcommand{\br}[3]{{$#1$}$\lower4pt\hbox{$\tp\atop\raise4pt \hbox{$\scriptscriptstyle{#2}$}$} ${$#3$}}
\newcommand{\tw}[3]{{$#1$}${\,\scriptscriptstyle {#2}}\atop\raise9pt\hbox{$\scriptstyle\tp$} ${$#3$}}
\newcommand{\ttps}[2]{{#1}\raise5pt\hbox{$\lower12pt\hbox{$\scriptstyle\tp$}\atop \lower0pt\hbox{$\tilde\;$}$}\raise4.5pt\hbox{${\scriptstyle{#2}}$}}
\newcommand{\st}[1]{\mbox{${\,\scriptscriptstyle {#1}}\atop\raise5.5pt\hbox{$*$}$}}

\newcommand{\rd}[1]{\mbox{${\,\scriptscriptstyle {#1}}\atop\raise5.5pt\hbox{$\bullet$}$}}
\newcommand{\rt}[1]{\otimes_\chi}
\newcommand{\lt}[1]{\mbox{${\,\scriptscriptstyle {#1}}\atop\raise5.5pt\hbox{$\ltimes$}$}}
\newcommand{\btr}{\raise1.2pt\hbox{$\scriptstyle\blacktriangleright$}\hspace{2pt}}
\newcommand{\btl}{\raise1.2pt\hbox{$\scriptstyle\blacktriangleleft$}\hspace{2pt}}

\newcommand{\lcr}{\raise1.0pt \hbox{${\scriptstyle\rightharpoonup}$}}
\newcommand{\rcr}{\raise1.0pt \hbox{${\scriptstyle\leftharpoonup}$}}

\newcommand{\ttp}{{\lower12pt\hbox{$\tp$}\atop \hbox{$\tilde\;$}}}

\newcommand{\id}{\mathrm{id}}
\newcommand{\im}{\mathrm{im}\:}

\newcommand{\kb}{\boldsymbol{k}}

\newcommand{\Ru}{\mathcal{R}}

\newcommand{\E}{\mathcal{E}}

\newcommand{\Mc}{\mathcal{M}}

\newcommand{\Q}{\mathcal{Q}}
\renewcommand{\O}{\mathcal{O}}

\newcommand{\C}{\mathbb{C}}
\newcommand{\Z}{\mathbb{Z}}

\newcommand{\tp}{\otimes}
\newcommand{\vt}{\vartheta}

\newcommand{\ve}{\varepsilon}
\newcommand{\gm}{\gamma}
\newcommand{\dt}{\delta}

\newcommand{\op}{\oplus}
\newcommand{\la}{\lambda}

\newcommand{\End}{\mathrm{End}}

\newcommand{\Ind}{\mathrm{Ind}}

\newcommand{\Tr}{\mathrm{Tr}}

\newcommand{\Rm}{\mathrm{R}}

\newcommand{\diag}{\mathrm{diag}}

\newcommand{\ad}{\mathrm{ad}}

\newcommand{\La}{\Lambda}

\newcommand{\g}{\mathfrak{g}}
\renewcommand{\b}{\mathfrak{b}}
\renewcommand{\k}{\mathfrak{k}}

\newcommand{\h}{\mathfrak{h}}
\newcommand{\mub}{\boldsymbol{\mu}}

\newcommand{\nb}{\boldsymbol{n}}
\newcommand{\mb}{\boldsymbol{m}}

\newcommand{\s}{\mathfrak{s}}
\newcommand{\n}{\mathfrak{n}}

\newcommand{\m}{\mathfrak{m}}

\newcommand{\nn}{\nonumber}

\newcommand{\p}{\mathfrak{p}}
\renewcommand{\l}{\mathfrak{l}}
\renewcommand{\c}{\mathfrak{c}}

\newcommand{\Spec}{\mathrm{Spec}}

\newcommand{\si}{\sigma}
\newcommand{\al}{\alpha}

\newcommand{\bt}{\beta}

\newcommand{\be}{\begin{eqnarray}}
\newcommand{\ee}{\end{eqnarray}}

\newtheorem{thm}{Theorem}[section]
\newtheorem{propn}[thm]{Proposition}
\newtheorem{lemma}[thm]{Lemma}

\newtheorem{definition}[thm]{Definition}

\newcount\prg

\newcommand{\parag}{\advance\prg by1 {\noindent\bf\thesection.\the\prg\hspace{6pt}}}

\begin{document}
\title{Non-Levi closed conjugacy classes of $SP_q(2n)$}
\author{
Andrey Mudrov \vspace{20pt}\\
\small Department of Mathematics,\\ \small University of Leicester, \\
\small University Road,
LE1 7RH Leicester, UK\\
\small e-mail: am405@le.ac.uk\\
}

\date{}
\maketitle

\begin{abstract}
We construct an explicit quantization of semi-simple conjugacy classes of the complex symplectic group
$SP(2n)$ with non-Levi isotropy subgroups through an operator realization
on  highest weight modules over the quantum group $U_q\bigl(\s\p(2n)\bigr)$.
\end{abstract}


\section{Introduction}
We construct a quantization of closed conjugacy classes of the complex algebraic group $SP(2n)$ whose
isotropy subgroup is not of Levi type. Such classes are not isomorphic to adjoint orbits in the Lie
algebra $\s\p(2n)$, and their Poisson structure  is not exactly $SP(2n)$-invariant.
The quantization features  a quantum group symmetry, which is a deformation of the
conjugation action of $SP(2n)$  on itself. The present study is based on \cite{M2} and
develops the ideas of \cite{M3}, where the simplest non-Levi conjugacy class $SP(4)/SP(2)\times SP(2)$ has been worked
out in details.

The conjugacy classes of interest form a family that is as large as of Levi type:
they involve diagonalizable symplectic matrices whose eigenvalues simultaneously include  $+1$ and $-1$
(a Levi class may have at most one of them).
Note that among the classical matrix groups only symplectic and orthogonal groups admit classes of this
type: for the special linear group they are  all isomorphic to adjoint orbits in the Lie algebra
and have Levi isotropy subgroups. In the present article we address only symplectic groups. Orthogonal
groups are given a special treatment in \cite{M5} based on a similar approach. Besides the basic similarities,
there are certain technical distinctions, and we have chosen to separate the orthogonal case from symplectic in order to simplify
the presentation.

The Poisson structure on the conjugacy classes comes from a Poisson structure on the group, which is
analogous to the canonical invariant Poisson structure on the Lie algebra $\g=\s\p(2n)$ (we assume the natural isomorphism
between the adjoint and coadjoint representations of $\g$). Quantization of
this structure is analogous to quantization of the Kostant-Kirillov-Souriau bracket
on the coadjoint orbits, with the difference that the former allows for quantum
group symmetry rather than classical.

Quantization of conjugacy classes with Levi isotropy subgroups has been constructed in \cite{M2} using the representation theory
of quantum groups. We should stress that the methods of \cite{M2} are inapplicable, as they are, for the non-Levi classes, whose
quantization is still an open problem. In our recent paper \cite{M3} we have shown how to approach it on
the simplest example of $SP(4)/SP(2)\times SP(2)$. In this work, we develop those ideas further and cover all
non-Levi conjugacy classes of $SP(2n)$. Along with the Levi type worked out in \cite{M2}, this is solving the problem
for all diagonalizable classes of $SP(2n)$.

Let us  we explain our methods.
It is natural to seek a quantization of an affine variety in terms of generators and relations, in other words,
as a quotient of a free algebra. Supposedly this projection factors through a projection
from a quantized coordinate ring $\C_\hbar[G]$ of the group  $G=SP(2n)$. It is an equivariant quantization
of the Poisson structure on $G$ whose restriction to conjugacy classes is the Poisson structure of our
interest. The algebra $\C_\hbar[G]$ is well studied
and its explicit description in generators and relations is available. It is related to the "reflection
equation" or "boundary Yang-Baxter equation", which is well established in the mathematical physics literature, \cite{KSkl,Mj}.
To ensure that the quotient of $\C_\hbar[G]$ is a flat deformation, we seek to realize it
in an algebra that is flat over the ring of formal power series in the deformation parameter $\hbar$.
Due to  certain structural properties of $\C_\hbar[G]$, this would also yield the defining relations,
provided we have managed to find an ideal in the kernel turning into the defining ideal in the classical
limit (such an ideal shall
automatically coincide with the kernel).

The algebra $\C_\hbar[G]$ can be also realized as a subalgebra in
the quantized universal enveloping algebra $U_\hbar(\g)$. This fact enables one
to construct the quantization of classes through
a realization of $\C_\hbar[G]$ in the algebra of endomorphisms of an appropriate $U_\hbar(\g)$-module.
This approach was successfully applied to conjugacy classes with Levi isotropy subgroups, which were quantized via
 parabolic Verma modules.
However, there is no immediate analog
of parabolic Verma modules for non-Levi subalgebras in $\g$. The obstructions are two-fold. Firstly, there is
no natural candidate for the quantized stabilizer as a subalgebra in $U_\hbar(\g)$. Secondly, even in the classical
case there is no parabolic extension of the non-Levi stabilizer. These are the principal properties that facilitate
the parabolic induction in the Levi case.
Therefore the key step is to find a suitable replacement of the parabolic Verma modules for non-Levi conjugacy classes.
 We take a quotient of a special auxiliary parabolic Verma module for it, which is chosen as follows.

Let $K\subset G$ denote the stabilizer of the initial point of the class. It contains a maximal
Levi subgroup $L\subset K$. There are actually two such subgroups, which correspond
to the two symplectic blocks $SP(2m)$ and $SP(2p)$ in $K$ rotating the $-1$-eigenspace and, respectively,
the $+1$-eigenspace of the initial point. We obtain $L$ by reducing  $SP(2m)$
 to $GL(m)$.
In the classical situation, the isotropy subalgebra $U(\k)\subset U(\g)$ is generated
over the Levi subalgebra $U(\l)$ by a certain pair of root vectors $e_{\dt}, f_{\dt}$.
We construct the parabolic Verma $U_\hbar(\g)$-module $\hat M_\la$  relative to $U_\hbar(\l)$,
where the highest weight $\la$ is conditioned by the presence of a singular vector  of weight $\la-\dt$.
The quotient $M_\la$ of $\hat M_\la$ over the  submodule generated by that singular vector is
the module where we realize the quantization of $\C[G/K]$.

Intuitively the passage from
$\hat M_\la$ to $M_\la$ can be interpreted as follows. The functional dimension of $\hat M_\la$
is $\frac{1}{2} \dim G/L$.
Its algebra of endomorphisms (locally finite part of)
is isomorphic to the tensor product $\hat M_{-\la}^*\tp \hat M_\la$, where $\hat M_{-\la}^*$ is the restricted
dual $U_\hbar(\g)$-module. It can be realized through a parabolic induction, and the pairing between
$\hat M_{-\la}^*$ and $\hat M_\la$ is a kind of Shapovalov form. The module $\hat M_{-\la}^*\tp \hat M_\la$ is
isomorphic to $\Ind_\l^\g \C$, where the induction is understood in the quantum group setting.
It  is a deformation of the classical induced module $\Ind_\l^\g \C$,
which is in duality with the function algebra on the coset space $G/L$; hence the idea to
realize the polynomial algebra $\C[G/L]$ in the self-dual module $\hat M_{-\la}^*\tp \hat M_\la$.  This qualitative consideration
explains why we should eliminate the part of $\hat M_\la$ generated by
$f_\dt v$. Doing so we kill the extra degrees of freedom along $\ad(\l) f_\dt$ and obtain a module, $M_\la$, with  the
proper functional dimension $\frac{1}{2}\dim G/K$.

The subalgebra $\C_\hbar[G]\subset U_\hbar(\g)$ is generated by the entries of an invariant
 matrix  $\Q\in \End(\C^{2n})\tp \C_\hbar[G]$ canonically
constructed  from the universal R-matrix of $U_\hbar(\g)$.
The problem of  the quantized ideal of the class boils down to determining the minimal polynomial of
 $\Q$ regarded
as an operator on $\C^{2n}\tp M_\la$.
The matrix $\Q$ is semi-simple on $\C^{2n}\tp \hat M_\la$, and its eigenvalues are known.
Clearly $\Q$ satisfies the same polynomial equation on $\C^{2n}\tp M_\la$, which
is however not necessarily minimal. We prove that the extra eigenvalue drops from the spectrum
of $\Q$ in the transition from $\hat M_\la$ to $M_\la$ and obtain the minimal polynomial
on  $\C^{2n}\tp  M_\la$ from that on $\C^{2n}\tp \hat M_\la$, through this reduction.

The passage from $\hat M_\la$ to $M_\la$ is analogous to the passage
from  to $G/L$  to  $G/K$, where the class $G/L$ is obtained from $G/K$ by splitting the $2m$ eigenvalues $-1$ into
$m$ pairs of reciprocals $\mu$, $\mu^{-1} \not= \pm 1$. In the limit as $\mu\to -1$ they glue
up, and the isotropy subgroup jumps from $L$ to $K$.
The minimal polynomial
of $G/L$ acquires a non-simple factor $(x+1)^2$, which should be reduced in the minimal polynomial
of $G/K$.
Similarly, we check that the extra divisor of the minimal polynomial of $\Q$
is canceled  in the projection   $\C^{2n}\tp \hat M_\la\to \C^{2n}\tp M_\la$, and the classical limit yields the minimal polynomial of $G/K$.
This implies the second important step of our strategy: the analysis of the $U_\hbar(\g)$-module $\C^{2n}\tp M_\la$
and the invariant operator $\Q$ on it.

Putting the non-Levi conjugacy classes into a common quantization scheme with the classes of Levi type implies several
far reaching consequences. First of all, recall that the latter
(along with quantum semi-simple coadjoint orbits) gave rise to the theory of dynamical Yang-Baxter equation
over a general non-Abelian base, \cite{DM}. To a large extent, that theory is based on the properties of the
parabolic $\O^\l$-category. Recall that it is a module category over that of finite dimensional representations of $U_\hbar(\g)$
under the tensor product multiplication (for a definition of module categories over monoidal categories, see e.g. \cite{O}).
We observe an analogous category $\O^\k$ associated with a non-Levi quantum conjugacy class,
which is generated (as a module category) by $M_\la$ of a feasible weight $\la$.
It is natural to expect that the study of $\O^\k$ will make a significant contribution
to the present theory of dynamical Yang-Baxter equation. Further, the parabolic category $\O^\l$ consists of $U_\hbar(\g)$-modules
that are parabolically induced from $U_\hbar(\l)$-modules.
At the same time, the algebra $U(\k)$ is not quantized as a Hopf subalgebra in $U_\hbar(\g)$, and
there is no {\em a priori} natural quantum counterpart for it.
It is therefore interesting
to understand its quantization, which would be a  $\k$-analog of the Levi subalgebra $U_\hbar(\l)$.
This might help to understand the category $\O^\k$.

An important special case of  non-Levi conjugacy classes comprises the symmetric spaces $SP(2n)/SP(2m)\times SP(2p)$, $m+p=n$.
There is an extended literature on their quantization in connection with
integrable models, \cite{FZ}, and representation theory, \cite{K,N,L3}.  The quantized function algebras were basically
viewed as subalgebras in the Hopf dual to $U_\hbar(\g)$ annihilated by certain
coideal subalgebras. Such subalgebras play the role of quantum stabilizers of the "initial point".
 An advanced
theory of quantum symmetric pairs (equivalently, quantum stabilizers) was developed in \cite{NS,L1,L2,L4}.
In the present paper, we adopt a different approach to quantization realizing it
by endomorphisms in a  $U_\hbar(\g)$-module.
The two approaches are complementary, as in the
classical geometry a closed conjugacy class can be alternatively presented as a subalgebra and a quotient
algebra of $\C[G]$. That is also possible in the special case of symmetric classes, because they possess  a "classical point", i.e.
a one-dimensional representation of $\C_\hbar[G]$. Equivalently, it is a numerical matrix solving the reflection
equation. Such a  matrix  is annihilated by the coideal subalgebra and  serves as the initial
point of the class in the theory of symmetric pairs. Contrary to our setting, this initial point is not diagonal.
As we already mentioned, our approach is lacking the quantum version of stabilizer, and that complicates  the further study
of the quantized non-Levi classes.
It is an interesting problem to match our approach with the theory of quantum symmetric pairs.
That could help to identify  the quantum stabilizer within the present approach and
 facilitate further advances in  the theory of quantized non-Levi conjugacy classes.

\section{Classical conjugacy classes}
\label{SecCCC}
Throughout the paper, $G$ designates the algebraic group $SP(2n)$ of symplectic matrices preserving
a non-degenerate skew symmetric form $||C_{ij}||_{i,j=1}^{2n}$ in the complex vector space $\C^{2n}$; the
Lie algebra of $G$ is denoted by $\g$.
We choose the realization  $C_{ij}=\epsilon_{i}\delta_{ij'}$, where
$\dt_{ij}$ is the Kronecker symbol,
$i'=2n+1-i$, and $\epsilon_i=-\epsilon_{i'}=1$ for  $i=1,\ldots, n$.

The polynomial ring $\C[G]$ is generated by the matrix coordinate functions $||A_{ij}||_{i,j=1}^{2n}$,
modulo the set of $2n\times 2n$ relations written in the matrix form as
\be
ACA^t=C.
\label{ideal_group}
\ee
The right conjugacy action of $G$ on itself induces a left  action on $\C[G]$ by duality; the matrix $A$  is invariant
as an element of $\End(\C^{2n})\tp \C[G]$.

The group $G$ is equipped with the Drinfeld-Sklyanin bivector field
\be
\{A_1,A_2\}=\frac{1}{2}(A_2A_1r-r A_1A_2),
\label{poisson_br_DS}
\ee
where $r\in \g\tp \g$ is a solution of the classical Yang-Baxter equation, \cite{D}.
Equation (\ref{poisson_br_DS}) is understood in $\End(\C^{2n})\tp \End(\C^{2n})\tp \C[G]$. The
subscripts label the natural embeddings of $\End(\C^{2n})$ in $\End(\C^{2n})\tp \End(\C^{2n})$, as usual in the
quantum groups literature.

The bivector field (\ref{poisson_br_DS}) defines a Poisson  bracket on
$\C[G]$ making $G$ a Poisson group.
There is a variety of solutions of the classical Yang-Baxter equation, which are parameterized by combinatorial objects (Manin triples)
and certain Cartan bivectos, \cite{BD}. We choose the so called standard solution
\be
r=\sum_{i=1}^{2n}(e_{ii}\tp e_{ii}-e_{ii}\tp e_{i'i'})
+2\sum_{i,j=1\atop i>j}^{2n}(e_{ij}\tp e_{ji}-\epsilon_{i}\epsilon_{j}e_{ij}\tp e_{i'j'}),
\label{r-matrix}
\ee
which is the simplest of all. The corresponding quantum group is pretty similar to the
classical universal enveloping algebra. In particular, one can define
quantum  Levi subalgebras, parabolic subalgebras {\em etc} facilitating  the parabolic
induction. The restriction to the standard solution (\ref{r-matrix}) is methodological. If one is concerned
with an abstract quantization, in terms of generators and relations rather than an operator realization,
the other cases can be readily obtained from the standard.
At the end of the article, we indicate what modifications to the resulting formulas should be made in order to include an
 arbitrary r-matrix.

We regard the group $G$ as a $G$-space under the  conjugation action.
The object of our study is another Poisson structure on $G$,
\be
\{A_1,A_2\}=\frac{1}{2}(A_2r_{21}A_1-A_1rA_2+A_2A_1r-r_{21}A_1A_2),
\label{poisson_br_sts}
\ee
in the shortcut matrix form. It is compatible with the conjugation action and makes $G$ a Poisson space
over the Poisson group $G$ equipped with the Drinfeld-Sklyanin bracket (\ref{poisson_br_DS}).
The bivector (\ref{poisson_br_sts}) restricts to every conjugacy class making it a Poisson homogeneous
space over $G$.

A closed conjugacy class $O\subset G$ consists of diagonalizable matrices and is
determined by the set of their eigenvalues
$S_O=\{\mu_i,\mu_i^{-1}\}_{i=1}^n$.
Every eigenvalue $\mu$ enters $S_O$ along with its reciprocal $\mu^{-1}$. In particular, there may be $\mu=\mu^{-1}=\pm 1$.
One should distinguish two situations: a) $S_O$ contains either $+1$ or $-1$ or none, and b) both $+1$ and $-1$ belong to $S_O$.
In the first case, $O$ is isomorphic to an orbit in $\g$ via the Cayley transformation, and its isotropy subgroup is
of Levi type. A conjugacy class of second type is not isomorphic to an adjoint orbit.
In terms of Dynkin diagram, every  Levi subgroup is obtained by scraping out
a subset of nodes, while for non-Levi isotropy
subgroups one should scrape out a set of nodes from the affine Dynkin diagram of $\g$.

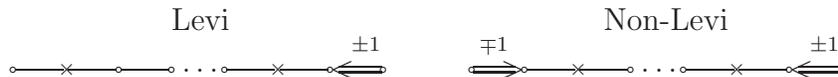
$$
\begin{picture}(150,40)
\put(60,25){Levi} \put(128,17){$\scriptstyle{\pm 1}$}
  \put(0,10){\circle{2}}  \put(1,10){\line(1,0){18}}
  \put(17,8){$\scriptstyle{\times}$}  \put(41,10){\line(1,0){18}}
  \put(40,10){\circle{2}}  \put(21,10){\line(1,0){18}}
  \put(60,10){\circle{2}}  \put(64,9.3){\ldots}
  \put(80,10){\circle{2}}  \put(81,10){\line(1,0){18}}
  \put(97,8){$\scriptstyle{\times}$}  \put(101,10){\line(1,0){18}}
  \put(120,10){\circle{2}}  \put(123.5,9){\line(1,0){16}}
  \put(140,10){\circle{2}}  \put(123.5,11){\line(1,0){16}}
   \put(120,7){$<$}
\end{picture}
\quad\quad
\begin{picture}(150,40)
\put(50,25){Non-Levi}
  \put(0,10){\circle{2}}  \put(0.5,9){\line(1,0){16}}\put(0.5,11){\line(1,0){16}}   \put(11,7){$>$}
  \put(20,10){\circle{2}}  \put(21,10){\line(1,0){18}}
  \put(37,8){$\scriptstyle{\times}$}  \put(41,10){\line(1,0){18}}
  \put(60,10){\circle{2}}  \put(64,9.3){\ldots}
  \put(80,10){\circle{2}}  \put(81,10){\line(1,0){18}}
  \put(97,8){$\scriptstyle{\times}$}  \put(101,10){\line(1,0){18}}
  \put(120,10){\circle{2}}     \put(120,7){$<$}\put(123.5,9){\line(1,0){16}}\put(123.5,11){\line(1,0){16}}
  \put(140,10){\circle{2}}
\put(128,17){$\scriptstyle{\pm 1}$}\put(3,17){$\scriptstyle{\mp 1}$}
\end{picture}
$$

Informally, a non-Levi centralizer necessarily contains two
symplectic blocks rotating the eigenspaces of eigenvalues  $\pm 1$.

We associate with a class $O$ an integer valued
vector $\nb=(n_1,\ldots, n_\ell, m,p)$ and a complex valued vector $\mub=(\mu_1,\ldots, \mu_\ell,-1,1)$
assuming $\mu_i$, $i=1,\ldots, \ell$, all invertible, not a square root of $1$, and $\mu_i\not =\mu_j^{\pm 1}$, $i\not =j$.
The initial point $o\subset O$ will be  fixed to the diagonal matrix with the entries
$$\underbrace{\mu_1,\ldots, \mu_1}_{n_1},\ldots, \underbrace{\mu_\ell,\ldots, \mu_\ell}_{n_\ell}
,\underbrace{-1,\ldots, -1}_{m},
\underbrace{1,\ldots, 1}_{2p},\underbrace{-1,\ldots, -1}_m,
\underbrace{\mu_\ell^{-1},\ldots, \mu_\ell^{-1}}_{n_\ell},\ldots, \underbrace{\mu_1^{-1},\ldots, \mu_1^{-1}}_{n_1},
$$
so that $\sum_{i=1}^\ell n_i+m+p=n$.
We reserve the integers  $m=n_{\ell+1},p=n_{\ell+2}$ to denote respectively, the ranks of the blocks
corresponding to $-1=\mu_{\ell+1}$ and $+1=\mu_{\ell+2}$ (we view $\pm1$ as degenerations of
the parameters $\mu_{\ell+1}$ and $\mu_{\ell+2}$).
The specialization  $n_1=\ldots =n_\ell=0$ is formally encoded by $\ell=0$ and
 referred to as the {\em symmetric case}, because the corresponding
conjugacy class is a symmetric space.

The stabilizer subgroup of the initial point $o$ is the direct product
\be
\label{gr_K}
K=GL(n_1)\times\ldots\times GL(n_\ell)\times SP(2m)\times SP(2p)
\ee
and it is determined by the vector $\nb$.
The positive integer $\ell$ counts the number of  the $GL$-blocks in $K$.
In the symmetric case, (\ref{gr_K}) reduces to $SP(2m)\times SP(2p)$.

Let $\Mc_K$ denote the moduli space of conjugacy classes with
the fixed isotropy subgroup (\ref{gr_K}), regarded as Poisson spaces.
The set of all  $\ell+2$-tuples $\mub$  with invertible components  such that
$\mu_{\ell+1}^2=\mu_{\ell+2}^2=1$
and $\mu_i\not= \mu_j^{\pm 1}$
for distinct $i,j$
parameterize $\Mc_K$ albeit not uniquely.
Multiplication by the scalar matrix $-1\in G$ preserves this set and swaps $\mu_{\ell+1}$ with $\mu_{\ell+2}$.
This transformation is an automorphism of $G$ as a  $G$-space
and preserves the Poisson structure (\ref{poisson_br_sts}). Therefore, the subset $\hat \Mc_K$ of $\mub$
with fixed $\mu_{\ell+1}=-1$ and $\mu_{\ell+2}=1$ can also be used for parametrization
of $\Mc_K$. The residual ambiguity is due to permutations of the components
$\mu_i\not =\pm 1$ with equal multiplicities.

The class $O$ associated with $\mub$ and $\nb$ is determined by the set of polynomial equations
\be
(A-\mu_1)\ldots (A-\mu_\ell)
(A+1)(A-1)(A-\mu_\ell^{-1})\ldots (A-\mu_1^{-1})=0,
\label{min_pol_cl}
\\
\Tr(A^k)=\sum_{i=1}^\ell n_i(\mu_i^{k}+\mu_i^{-k})+2m(-1)^k+2p, \quad k=1,\ldots, 2n,
\label{tr_cl}
\ee
on the entries of the matrix $A$.
In fact, the ideal in $\C[G]$ generated by this set of relations is radical
and therefore coincides with the defining ideal of $\C[O]$ in $\C[G]$. This is a consequence of the following general fact.

Consider a smooth variety $X$  in an affine space $Y$ of dimension $\dim(Y)$.
Suppose that $X$ is defined by a system of polynomial equations $F_i(x)=0$, $i\in I$, where
$I$ is a finite set of indices. The ideal  $J'=(F_i)_{i\in I}$ is contained in the
defining ideal  $J$ of $X$,  i.e. the ideal of all polynomial functions vanishing on $X$. In general,
$J'$ might be less than $J$, and then the quotient $\C[Y]/J'$ cannot be regarded
as a ring of functions on $X$: there will be nilpotent elements in $\C[Y]/J'$. It is essential for our
approach to quantization to ensure  that the ideal $J'$ is exactly $J$. The latter obeys
a certain maximality requirement, which plays a role in the construction.
We will use the following criterion of radicality of $J'$.
\begin{propn}
\label{radical_gen}
Suppose that the rank of the differential $\{dF_i\}_{i\in I}$
is  equal to $\dim(Y)-\dim(X)$  at every point $x\in X$. Then the ideal $J'\subset \C[Y]$ generated by $\{F_i\}_{i\in I}$
coincides with  the defining ideal $J$ of $X$.
\end{propn}
\begin{proof}
Denote by   $A'=\C[Y]/J'$ and  $A=\C[Y]/J$  the quotient algebras and consider their affine schemes with the structure
sheafs $\mathcal{O}'$ and $\mathcal{O}$, respectively. Since
$J$ is the radical of $J'$, the natural embedding  $\Spec(A')\to
\Spec(A)$ is an isomorphism making $\mathcal{O}$  a
subsheaf in $\mathcal{O}'$. The condition on the rank of $\{dF_i\}_{i\in I}$ implies, by the Jacobian criterion of smoothness
\cite{Eis},
that $\mathcal{O}_a$ and $\mathcal{O}_a'$ are regular local rings at every point $a\in
\Spec(A')$, and $\mathcal{O}_a=\mathcal{O}_a'$. As the two sheafs coincide locally,
they coincide globally. Hence $A'\simeq A$, and $J'=J$.
\end{proof}

Proposition \ref{radical_gen} provides a convenient test for verification if a particular system of equations
gives rise to the defining ideal. That is especially so for homogeneous varieties, as it suffices to look at the initial point only. Remark that
the condition on the rank
can be replaced with a more practical condition on the kernel: $\dim( \cap_{i}\ker dF_i)=\dim(X)$.

First we consider the general linear group, which result was obtained in \cite{P} within a different approach.
\begin{thm}
\label{prop_clas_gl}
Let $\tilde G$ be the general linear group of the vector space $\C^N$ and
let $\tilde O\subset \tilde G$ be the conjugacy class  of diagonalizable matrices with pairwise
distinct eigenvalues
$(\mu_1,\ldots, \mu_l)$ of multiplicities $(n_1,\ldots, n_l)$.
The system of polynomial equations
\be
(A-\mu_1)\ldots (A-\mu_l)=0, \quad \Tr(A^k)=\sum_{i=1}^l n_i \mu_i^k, \quad k=1, \ldots, N
\label{gl_class}
\ee
has rank $N^2-\dim(\tilde O)$ everywhere on $\tilde O$, hence it generates the defining ideal of $\tilde O$.
\end{thm}
\begin{proof}
Take $o=\diag(\underbrace{\mu_1,\ldots,\mu_1}_{n_1},\ldots, \underbrace{\mu_l,\ldots,\mu_l}_{n_l})\in \tilde O$ for
 the initial point in
$\tilde O$.
The matrix $o$ can be  written
as $o=\sum_{i=1}^l \mu_i P_i$, where
$P_i \colon \C^N\to \C^{n_i}$ is the diagonal projector to the $\mu_i$-eigenspace of $o$.
 Denote by $E_{ij}$, $i,j=1,\ldots, l$, the
subspace of matrices $P_i \End(\C^{N}) P_j$.
We have $\End(\C^N)=\oplus_{i,j=1}^l E_{ij}$ for the entire matrix algebra, and
$\tilde \k=\oplus_{i=1}^l E_{ii}$ for the Lie algebra $\tilde \k$  of the stabilizer
of $o$. The tangent space $T_o(\tilde G)$ is naturally identified with
$\tilde \m=\oplus_{ i,j=1\atop i\not =j}^l E_{ij}$. The class $\tilde O$ is the zero locus of the
system of equations (\ref{gl_class}).
To prove the statement, it is sufficient to check the rank of the system (\ref{gl_class}) at the  point $o$.

Denote by $F$ the matrix polynomial $\prod_{i=1}^l(A-\mu_i)$ and
by $\vt_k$ the trace $\Tr(A^k)$, $k=1, \ldots, N$.
The system of relations involves  $N\times  N$ matrix entries $F_{ij}$ and $N$ differences
$\vt_k-\sum_{i=1}^l n_i \mu_i^k$.
It is easy to check that
$$
dF(\xi)=0, \quad d\vt_k(\xi)=0,\quad k=1,\ldots, N,
$$
for all $\xi\in E_{ij}$  with $i\not=j$ and
$$
dF(\xi)=\prod_{i=1\atop i\not =j}^l(\mu_j-\mu_i)\>\xi, \quad
d\vt_k(\xi)=k \mu_j^k\Tr(\xi) ,\quad k=1,\ldots, N,
$$
for all $\xi\in E_{jj}$.
Note that the equations with $d\vt_k$ are redundant as
$\ker (dF)\subset \ker d\vt_k$. To see this, one should differentiate the trace of $F(\xi)$.
The left equation  tells us that
$\im dF=\tilde \k$, as the numerical coefficient before $\xi$ does not vanish. This proves the assertion.
\end{proof}

Based on Theorem \ref{prop_clas_gl}, we apply Proposition \ref{radical_gen} to describe the
defining ideals of closed conjugacy classes of the symplectic groups.

\begin{thm}
The system of polynomial relations (\ref{min_pol_cl}) and (\ref{tr_cl}) along
with the defining relations of the group (\ref{ideal_group})  generate the defining ideal of the
class $O\subset SP(2n)$.
\label{prop_clas_sp}
\end{thm}
\begin{proof}
As shown in the proof of Theorem \ref{prop_clas_gl}, the differential
on the trace functions is linear dependent of the differential of
the minimal polynomials. Therefore, the essential part of the Jacobian
comes from the minimal polynomial and the equation of the group.
 The tangent space $T_o(G)$ is the set of fixed points
of the linear automorphism  $\si_o\colon \xi\mapsto -o\si(\xi)o$, where $\si$ is the  involutive
algebra anti-automorphism $\xi\mapsto -C \xi^t C$. Clearly
the tangent space can be presented as $o\g$ using the matrix multiplication and
the embedding $\g\subset \End(2n)$.

The automorphism
$\si_o$ is an involution, so the tangent space $T_o(G)$  is the image of the idempotent
 $\frac{\id+\si_o}{2}$.
Using the same notation as in the proof of Theorem \ref{prop_clas_gl}, we
split the tangent space  into the  direct sum $o\g=\k_o\oplus \m_o$,
where $\k_o=o\g\cap \tilde \k$ and $\m_o=o\g\cap \tilde \m$.
This is possible because the spaces $\tilde \k$ and $\tilde \m$ are
stable under multiplication by $o$.

Let $H(A)$ denote the endomorphism $\End(\C^{2n})\to \End(\C^{2n})$, $A\mapsto ACA^tC+1$.
We need to find the rank of the differential of the mapping $\End(\C^{2n}) \to \End(\C^{2n}) \oplus \End(\C^{2n})$,
$A\mapsto H(A)\oplus F(A)$, at the point $A=o$. Equivalently we can find its kernel,
which is the intersection $\ker dH_o\cap \ker dF_o$.
The tangent space $T_o\bigl(\End(\C^{2n})\bigr)$ splits into the direct sum
$o\g^\perp\oplus o\g$. The kernel of  $dH_o$ is exactly $o\g$, hence
 $\ker(dH_o\oplus dF_o)$ is just $\ker dF_o|_{o\g}$.
In the course of the proof of Theorem \ref{prop_clas_gl}  we saw that $\m_o\subset  \ker dF_o|_{o\g}$. This
inclusion is, in fact, an equality.
Indeed, $\k_o\subset \tilde \k$, and  $dF_o$ is injective on $\tilde \k$.
Hence it is injective on $\k_o$. Thus,
the kernel of the differential $dH_o\oplus dF_o$ is exactly $\m_o$.
But $\m_o\simeq T_o(O)$, and the rank of the map
$dH_o\oplus dF_o\colon \End(\C^{2n})\to \End(\C^{2n})\oplus \End(\C^{2n})$ is equal
to the codimension of $G$. This completes the proof.
\end{proof}
Although non-Levi conjugacy classes are of our main concern,
Theorem \ref{prop_clas_sp} holds true for any semi-simple conjugacy class.
It generalizes for the orthogonal groups in the obvious way, with the only stipulation for the $D$-series:
the traces of matrix powers are not enough to fix a class, and one
has to add one more condition on the invariants of $G$, see e.g. \cite{M2}.

\section{Quantum group $U_\hbar\bigl(\s\p(2n)\bigr)$}
\label{secQG}
Recall the definition of the quantum group
$U_\hbar\bigl(\s\p(2n)\bigr)$, which is a deformation of the universal enveloping
algebra $U\bigl(\s\p(2n)\bigr)$ along the formal parameter $\hbar$ in the class of Hopf algebras, \cite{D}.
Let $R$ and  $R^+$ denote respectively the root system and the subset of positive roots of the Lie algebra $\g=\s\p(2n)$.
Let  $\Pi_+=(\al_1,\al_1,\ldots, \al_{n})$ be the set of simple positive roots.
By $(\>.\>,\>.\>)$ we designate the canonical inner form on the linear span of $\Pi^+$.
The set $\Pi^+$ can be  conveniently expressed through an orthogonal basis $(\ve_i)_{i=1}^n$ by
$\al_i=\ve_i-\ve_{i+1}$, $i=1,\ldots, n-1$, $\al_n =2\ve_n$. We reserve the special notation
$\bt$ for the long root $\al_n$.

The inner product establishes
a linear isomorphism between the linear span $\C\Pi^+$ and its dual, $\h$.
We define $h_\la\in \h$
for every $\la\in \h^*=\C\Pi^+$ as the image of $\la$ under this isomorphism:  $\mu(h_\la)=(\la,\mu)$ for all $h\in \h$.
In particular, we set $h_\rho$ for  the half-sum of all positive roots $\rho=\frac{1}{2}\sum_{\al\in \Rm_+}\al $.

The quantum group $U_\hbar(\g)$ is a $\C[\![\hbar]\!]$-algebra generated by the simple root vectors (Chevalley generators)
$e_{\mu}$, $f_{\mu}$,
and the Cartan generators $h_{\mu}\in \h$,
$\mu\in \Pi^+$.
The vector space $\h$  generates the commutative
Cartan  subalgebra $U_\hbar(\h)$ in $U_\hbar(\g)$. The elements $h_\mu$  obey the following commutation
relations with $e_{\nu}$, $f_{\nu}$:
$$
[h_{\mu},e_{\nu}]= (\mu,\nu) e_{\nu},
\quad
[h_{\mu},f_{\nu}]= -(\mu,\nu) f_{\nu},
\quad \mu,\nu\in \Pi^+.
$$
In these formulas, the only non-zero inner products are
$$
(\al_i,\al_i)=2, \quad (\al_{i-1},\al_i)=-1,
\quad
(\bt,\bt)=4, \quad (\bt,\al_{n-1})=-2,
$$
where $i$ takes all admissible values in the range $1,\ldots,n-1$.
Note that the Cartan generators $h_\mu$ are different from those of \cite{ChP}, which are obtained
from $h_\mu$ via division by $\frac{(\mu,\mu)}{2}$.

The positive  and negative Chevalley generators commute to $U_\hbar(\h)$:
$$
[e_{\mu},f_{\nu}]=\delta_{\mu,\nu} \frac{q^{h_{\mu}}-q^{-h_{\mu}}}{q_\mu-q^{-1}_\mu},\quad \mu \in \Pi^+,
$$
where $q_\mu= q=e^\hbar$ for $\mu\not =\bt$ and $q_\bt= q^2$.

The non-adjacent positive Chevalley generators commute. The adjacent generators satisfy
the Serre relations
$$
e_{\mu}^{2}
e_{\nu}
-
(q+q^{-1})
e_{\mu}e_{\nu}e_{\mu}
+
e_{\nu}e_{\mu}^{2}=0, \quad\mbox{for }\quad \mu,\nu\not =\bt,  \quad\mbox{and }
$$
$$
e_{\bt}^{2}
e_{\mu}
-
(q^2+q^{-2})
e_{\bt}e_{\mu}e_{\bt}
+
e_{\mu}e_{\bt}^{2}=0,
$$$$
\quad
e_{\mu}^{3}
e_{\bt}
-
(q^2+1+q^{-2})e_{\mu}^{2}
e_{\bt}e_{\mu}
+
(q^2+1+q^{-2})e_{\mu}
e_{\bt}e_{\mu}^{2}
-
e_{\bt}e_{\mu}^{3}
=0
$$
for $\mu=\al_{n-1}$.
Similar relations holds for the negative Chevalley generators $f_\mu$.

The  involution $\omega\colon e_\mu\leftrightarrow f_\mu$ and  $\omega(h_\mu)=-h_\mu$, $\mu\in \Pi^+$, extends to an algebra
automorphism of $U_\hbar(\g)$

The comultiplication $\Delta$ and antipode $\gm$ are defined on the generators by
$$
\Delta(h_\mu)=h_\mu\tp 1+1\tp h_\mu, \quad \gm(h_\mu)=-h_\mu,
$$
$$
\Delta(e_\mu)=e_\mu\tp 1+q^{h_\mu}\tp e_\mu, \quad \gm(e_\mu)=-q^{-h_\mu}e_\mu,
$$
$$
\Delta(f_\mu)=f_\mu\tp q^{-h_\mu}+1\tp f_\mu, \quad \gm(f_\mu)=-f_\mu q^{h_\mu},
$$
for all $\mu\in \Pi_+$.
The counit homomorphism $\ve\colon U_\hbar(\g)\to \C[\![\hbar]\!]$ annihilates $e_\mu$, $f_\mu$, $h_\mu$.
Note that we use the opposite coalgebra structure on $U_\hbar(\g)$ as compared to \cite{ChP}.

Besides the Cartan subalgebra $U_\hbar(\h)$, the quantum group $U_\hbar(\g)$ contains the following Hopf subalgebras.
 The positive and negative
Borel subalgebras $U_\hbar(\b^\pm)$ are generated over $U_\hbar(\h)$ by
$\{e_\mu\}_{\mu\in \Pi^+}$ and $\{f_\mu\}_{\mu\in \Pi^+}$, respectively. For any Levi subalgebra $\l\subset \g$
corresponding to a subset $\Pi^+_\l\subset \Pi^+$, the universal enveloping algebra $U(\l)$ is quantized to a Hopf subalgebra
$U_\hbar(\l)\subset U_\hbar(\g)$, along with the parabolic subalgebras $U_\hbar(\p^\pm)$ generated by $U_\hbar(\b^\pm)$
over $U_\hbar(\l)$.

The triangular decomposition $\g=\n_\l^-\op \l \op \n_\l^+$ gives rise to the factorization
\be
\label{tr_fac}
U_\hbar(\g)=U_\hbar(\n_\l^-) U_\hbar(\l)U_\hbar(\n_\l^+),
\ee
where $U_\hbar(\n_\l^\pm)$ are certain subalgebras in $U_\hbar(\b^\pm)$, \cite{Ke}.
This factorization makes $U_\hbar(\g)$ a free $U_\hbar(\n_\l^-)-U_\hbar(\n_\l^+)$-bimodule generated by $U_\hbar(\l)$.
In the special case of this decomposition relative to $\l=\h$, we use the notation
$U_\hbar(\g_\pm)=U_\hbar(\n^\pm_\h)$. Note that, contrary to the classical situation,
$U_\hbar(\n^\pm_\l)$ are not Hopf subalgebras in $U_\hbar(\g)$.

We shall also deal with the  Hopf subalgebra $U_q(\g)\subset U_\hbar(\g)$ generated by the
simple root vectors and the exponentials $t^{\pm}_{\al_i}=q^{\pm h_{\al_i}}$, $\al_i\in \Pi_+$.
Let us stress that we regard $U_q(\g)$ as a $\C[\![\hbar]\!]$-algebra.
The other mentioned subalgebras of $U_\hbar(\g)$ have their counterparts in $U_q(\g)$, and
we use the subscript $q$ for their notation. The roles of quantum groups $U_\hbar(\g)$ and $U_q(\g)$
are different in what follows. While $U_q(\g)$ is a source of  non-commutative
functions on quantum geometric spaces, $U_\hbar(\g)$ is a measure of their symmetry. This difference
is somewhat camouflaged in the classical geometry but becomes more distinctive in quantum.

\section{Quantum subgroup $U_\hbar\bigl(\g\l(n)\bigr)$}
The quantum group $U_\hbar\bigl(\s\p(2n)\bigr)$ contains the quantum subgroup $U_\hbar\bigl(\g\l(n)\bigr)$
corresponding to the positive simple roots $(\al_1,\ldots, \al_{n-1})\subset \Pi^+$. We need a few technical facts about
this subalgebra, which are used in the sequel.

Fix a pair of integers $i,j$ such that $i<j<n$ and put $\mu=\al_i+\ldots +\al_j\in R^+$. Along with $(\al_1,\ldots, \al_{n-1})$,
such $\mu$ exhaust all of positive roots of $\g\l(n)$. The integer $j-i+1$ is called  height of the root $\mu$
and denoted by $\mathrm{ht}(\mu)$.
Define elements $f_{\mu},\tilde f_{\mu}\in U_\hbar\bigl(\g\l(n)\bigr)$ by
$$
f_{\mu}=[f_{\al_{i}},\ldots [f_{\al_{j-1}},f_{\al_{j}}]_q]_q\ldots ]_q,\quad
\tilde f_{\mu}=[f_{\al_{i}},\ldots [f_{\al_{j-1}},f_{\al_{j}}]_{q^{-1}}]_{q^{-1}}\ldots ]_{q^{-1}},
$$
were  $[x,y]_a$ designates  the combination $xy-ayx$ with a scalar $a$.
It is also convenient to put  $\tilde f_{\al_i}=f_{\al_i}$ for the simple roots $\al_i$, $i=1,\ldots, n-1$.
Here are some commutation relations involving these root vectors.
\begin{lemma}
\label{lem_aux1}
Let $\mu=\al_i+\ldots+ \al_j$ and $f_\mu$ and $f_{\tilde\mu}$ be as above. Suppose the integer $k$ is such that $i< k <j$. Then
$
[e_{\al_k},f_{\mu}]=0,\quad [e_{\al_k},\tilde f_{\mu}]=0.
$
Further,
$$
[e_{\al_i},f_{\mu}]=f_{\mu'}q^{-h_{\al_{i}}}
,\quad
[e_{\al_j},f_{\mu}]=-q f_{\mu''}q^{h_{\al_{j}}}
,\quad
[e_{\al_i},\tilde f_{\mu}]=
\tilde f_{\mu'}q^{h_{\al_{i}}}
,\quad
[e_{\al_j},\tilde f_{\mu}]=-q^{-1} \tilde f_{\mu''}q^{-h_{\al_{j}}},
$$
where $\mu'=\al_{i+1}+\ldots+\al_{j}$ and $\mu''=\al_{i}+\ldots+\al_{j-1}$.
\end{lemma}
\begin{proof}
It is sufficient to check only the group of equalities involving $f_\mu$, as the
equalities with $\tilde f_\mu$ can be obtained by the formal
replacement $q\to q^{-1}$.
Let us start with  the special case of $k=i+1$, $j=i+2$:
$$
[e_{\al_{i+1}},f_{\mu}]
\sim [f_{\al_{i}},[q^{h_{\al_{i+1}}}-q^{-h_{\al_{i+1}}},f_{\al_{i+2}}]_q]_q\sim
[f_{\al_{i}},f_{\al_{i+2}}q^{-h_{\al_{i+1}}}]_q=[f_{\al_{i}},f_{\al_{i+2}}]q^{-h_{\al_{i+1}}}=0.
$$
The general case is verified in a similar way based on the formula
$f_\mu=[f_{\mu_1},[f_{\al_k},f_{\mu_2}]_q]_q$, where  $\mu_1=\al_{i}+\ldots +\al_{k-1}$
and $\mu_2=\al_{k+1}+\ldots +\al_{j}$.
 This formula is an elementary corollary of the definition of $f_\mu$.
Further,
\be
[e_{\al_i},f_{\mu}]&=&[\frac{q^{h_{\al_{i}}}-q^{-h_{\al_{i}}}}{q-q^{-1}},f_{\mu'}]_q=-\frac{1}{q-q^{-1}}[q^{-h_{\al_{i}}},f_{\mu'}]_q=
f_{\mu'}q^{-h_{\al_{i}}},
\nn\\[1pt]
[e_{\al_j},f_{\mu}]&=&[f_{\mu''},\frac{q^{h_{\al_{i}}}-q^{-h_{\al_{i}}}}{q-q^{-1}}]_q=\frac{1}{q-q^{-1}}[f_{\mu''},q^{h_{\al_{j}}}]_q
=\frac{(1-q^2)}{q-q^{-1}}f_{\mu''}q^{h_{\al_{j}}}=-q f_{\mu''}q^{h_{\al_{j}}},
\nn
\ee
as required.
\end{proof}
In particular, a positive Chevalley generator $e_\al$ commutes with $f_\mu$,  $\mathrm{ht}(\mu)>1$, unless
$\mu-\al$ is a root. If $\mu-\al$ is a root, then $\al$ is either the smallest or the greatest simple root
in $\mu$ (under the natural ordering of $\{\al_i\}_{i=1}^{n-1}$). The commutator $[e_\al,f_\mu]$ includes the
factor
$f_{\mu-\al}$ in both cases, but we find it convenient to keep reference to the position
of $\al$ within $\mu$  using notation $\mu'$ and $\mu''$.
\begin{lemma}
Suppose $\mu=\al_i+\ldots+ \al_j$ and  $g$ is a monomial (word)  in the simple root vectors $\{f_{\al_k}\}_{k=i}^j$
that contains $f_{\al_i}$ and $f_{\al_j}$ at most once.
Then
\begin{enumerate}
\item $[g,f_\mu]=0$ if both $f_{\al_i}$ and $f_{\al_j}$ enter $g$ or none,
\item $[g,f_\mu]_{q^{-1}}=0$ if $g$ contains only $f_{\al_i}$,
\item $[g,f_\mu]_{q}=0$ if $g$ contains only  $f_{\al_j}$.
\end{enumerate}
In particular,
$
[\tilde f_\mu,f_\mu]=0,
$
$
[\tilde f_{\mu'}, f_\mu]_q=0,$
and
$
[\tilde f_{\mu''}, f_\mu]_{q^{-1}}=0.
$
\label{[tilde,nontilde]}
\end{lemma}
\begin{proof}
It is known that $f_{\al_k}$ commutes with $\tilde f_{\gm}$ if $i<k<j$, see e.g. \cite{M4}.
Further, the "higher order Serre relations"
\be
f_{\al_i}f_\mu&=&f_{\al_i}[f_{\al_i},f_{\mu'}]_q=q^{-1}[f_{\al_i},f_{\mu'}]_qf_{\al_i}=q^{-1}f_{\mu}f_{\al_i},
\nn\\
f_{\al_j}f_\mu&=&f_{\al_j}[f_{\mu''},f_{\al_j}]_q=q[f_{\mu''},f_{\al_j}]_q f_{\al_j}=qf_\mu f_{\al_j}.
\nn
\ee
in $U_\hbar\bigl(\g\l(n)\bigr)$ readily imply the statement.
\end{proof}

\section{Generalized Verma module $M_\la$}
\label{secGVMM}
We need to set up a few conventions about representations of quantum groups.
We assume that they are free modules over the ring of scalars, and their rank will
be referred to as  dimension.
We call a $U_\hbar(\g)$-module irreducible if it is irreducible over $U(\g)$ in the classical limit.
As $U_q(\g)$ and $U_\hbar(\g)$ have different Cartan subalgebras, their sets of weights are different.
Still we prefer to use additive language for $U_q(\g)$-weights, which are then parameterized
by the assignment $\la\mapsto q^\la$.

We shall be dealing with weight modules over $U_\hbar(\g)$ generated by a weight vector
$v$ annihilated by  the positive
Chevalley generators. Under our convention, all weights in such modules belong to $-\Z\Pi^++\la$, where
$\la\in \h^*[\![\hbar]\!]$ is the highest weight supported by $v$. However, such representations are not sufficient for us,
because they yield $\lim_{q\to 1} q^{2 \la}=1$. In our  further constructions,
$\lim_{q\to 1} \diag(q^{2(\la,\ve_1)},\ldots ,q^{2(\la,\ve_n)},q^{-2(\la,\ve_n)},\ldots ,q^{-2(\la,\ve_1)})\in G$ is
 the initial point of  the conjugacy class under study. Therefore, we should include highest weights
from $\hbar^{-1}\h^*[\![\hbar]\!]$. These are actually weights  of $U_q(\g)$
as they are well defined on $q^{\pm h_{\al_i}} \in U_q(\g)$.
It is sufficient for our purposes to restrict $\la$ to $\hbar^{-1}\h^*\oplus \h^*\subset \hbar^{-1}\h^*[\![\hbar]\!]$.

The presence of the $\hbar$-irregular term in  $\hbar^{-1}\h^*\op \h^*$ implies that the representation of $U_q(\g)$ on
such a module, $V$, cannot be extended to a representation of  $U_\hbar(\g)$. We have to extend $V$ over the field of Laurent
series $\C(\!(\hbar)\!)$ (bounded from below), where the action of $U_\hbar(\g)$ can
be defined. The corresponding extension $\End(V)(\!(\hbar)\!)$ contains $\End(V)$ as a $\C[\![\hbar]\!]$-submodule. It turns out that the locally finite part of $\End(V)$, which is
still a $\C[\![\hbar]\!]$-submodule, does admit a $U_\hbar(\g)$-action, cf. Section \ref{secQCC}.

Let $L\subset K$ denote the Levi subgroup
$$
L=GL(n_1)\times \ldots \times GL(n_\ell)\times GL(m)\times SP(2p)
$$
for  $K$ as in  (\ref{gr_K}).
The difference between $L$ and $K$ is only one Cartesian factor $GL(m)\subset SP(2m)$. By $\l$ we denote
the Lie algebra of $L$. It is a reductive subalgebra in $\g$ of maximal rank $n$.

By $\c_\l\subset \h$ we denote the centre of $\l$.
In the presence of the canonical inner product on $\h^*$, we identify its dual $\c_\l^*$ with the subspace in $\h^*$
which is orthogonal to the annihilator of $\c_\l$.
Any element $\la\in \hbar^{-1} \c_\l^*\op  \c_\l^*$ defines
a one-dimensional representation of $U_q(\l)$ denoted by  $\C_\la$.
It assigns nil to the Chevalley generators of $U_q(\l)$ and acts by $q^{h_\al}\mapsto q^{(\la,\al)}$ on the Cartan subalgebra.
This representation extends to the parabolic subalgebra $U_q(\p^+)$ as trivial on
 $U_q(\b^+)\subset U_q(\p^+)$.
Denote by $\hat M_\la=U_q(\g)\tp_{U_q(\p^+)}\C_\la$ the parabolic Verma $U_q(\g)$-module induced from $\C_\la$.
It is generated by the highest weight vector, which we denote by $v$.
We impose the condition  on $\la$   that $\hat M_\la$ admits a singular vector of weight $-\dt+\la$, where
$\dt=2\al_{n-p}+\ldots+ 2\al_{n-1}+\bt\in R^+$. Such a vector generates a submodule $\hat M_{\la-\dt}\subset \hat M_{\la}$.
The quotient module $\hat M_\la/\hat M_{\la-\dt}$ is the subject of our interest.

We consider $\hat M_\la/\hat M_{\la-\dt}$ as a candidate in order to realize  the quantization
of $G/K$ in $\End(\hat M_\la/\hat M_{\la-\dt})$, by the analogy with the Levi class $G/L$.
The tangent space to  $G/L$ at the initial point is naturally identified with $\n^+_\l\oplus \n^-_\l$.
 The locally finite part of $\End(\hat M_\la)$ is isomorphic to $\hat M_\la^*\tp \hat M_\la$,
 where $\hat M_\la^*$ is the dual parabolic Verma module $U_q(\g)\tp_{U_q(\p^-)}\C_{-\la}$.
For generic $\la$, a quantization of $G/L$ can be realized in $\End(\hat M_\la)$,
\cite{M2}. As a vector space, $\End(\hat M_\la)$ is isomorphic to $ U_\hbar(\n^+_\l)\tp U_\hbar(\n^-_\l)$,
and the functional dimension of $\hat M_\la$ is equal to $\dim \n^-_\l=\frac{1}{2}\dim G/L$.
 The tangent space at the initial point
 of $G/K$ is transversal to   $\m^-\oplus \m^+$, where $\m^-=\ad(\l)(f_\dt)$ and
 $\m^+=\ad(\l)(e_\dt)$. It is presentable as
 $\n^+_\k\oplus \n^-_\k$ on setting $\n^\pm_\k=\n^\pm_\l \ominus \m^\pm$. A
 module that supports quantization
 of $G/K$ should have the functional dimension $\dim \n^-_\k=\frac{1}{2}\dim G/K$.
Such is the quotient $\hat M_\la/\hat M_{\la-\dt}$,  where the  vector $f_\dt$  vanishes along with a
$q$-version of $\m^-$,
cf. Section \ref{SecPBW}.

For the sake of technical convenience, we assume that $\ell=0$, $m=1$, $n=1+p$. This restriction will be relaxed later on.
In this setting, the root $\al_1$ is distinguished, as $f_{\al_1}$ is the only negative
Chevalley generator which does not belong to $U_q(\l)$ and does not kill the highest weight vector  $v\in \hat M_\la$.

\begin{lemma}
\label{pre_dt'}
For all $i=1,\ldots,p$ and  $k=2,\ldots ,p+1$, $k\not = i+1$, the vector $f_{\al_k} f_{\al_i}\stackrel{>}{\ldots} f_{\al_1}v\in \hat M_\la$
is nil.
\end{lemma}
\begin{proof}
We prove it by induction on $i$. The case $i=1$ is obvious,
as $f_{\al_k}$ commutes with $f_{\al_1}$ for all $k=3,\ldots, p+1$. Suppose we have done
it for all $i=1,\ldots, l-1$ and let us prove it for $i=l$. Consider the following three cases first.
\begin{itemize}
\item[a)] $k\geqslant l+2$. The generator $f_{\al_k}$ can be freely pushed to the right until it kills $v$.
\item[b)] $k=l$. Plugging $f_{\al_l}^2f_{\al_{l-1}}=(q+q^{-1})f_{\al_l}f_{\al_{l-1}}f_{\al_l}-f_{\al_{l-1}}f_{\al_l}^2$
into $f_{\al_l} f_{\al_l}\stackrel{>}{\ldots} f_{\al_1}v$ moves at least one copy of $f_{\al_l}$ to the right of $f_{\al_{l-1}}$.
This proves b) by reducing it to  a) for $i=l-2$ and applying the induction assumption.
\item [c)] $k=l-1$. Plug
$f_{\al_{l-1}}f_{\al_{l}}f_{\al_{l-1}}=\frac{1}{q+q^{-1}}(f_{\al_{l-1}}^2f_{\al_{l}}+f_{\al_{l}}f_{\al_{l-1}}^2)$
into $f_{\al_{l-1}} f_{\al_l}f_{\al_{l-1}}\stackrel{>}{\ldots} f_{\al_1}v$.
The first summand falls into a) for $i=l-2$, and the second into the b) for $i=l-1$. They both vanish by the induction assumption.
\end{itemize}
Finally, if $k<l-1$, then $f_{\al_k}$ is pushed to the right till it meets $f_{\al_{k+1}}$. This falls into the case c) for
$i=k+1$ and vanishes by the induction assumption.
\end{proof}
We apply Lemma \ref{pre_dt'} to analyze the structure of certain weight subspaces in $\hat M_\la$.
\begin{lemma}
\label{dt'}
Put $\delta'=\al_1+2\al_2+\ldots +2\al_p+\bt$. The subspace of weight
$-\dt'+\la$ in $\hat M_\la$ is spanned by
the vector
$
f_{\al_2}\stackrel{<}{\ldots} f_{\al_p}f_{\bt}f_{\al_p}\stackrel{>}{\ldots}f_{\al_2}f_{\al_1}v.
$
\end{lemma}
\begin{proof}
The subspace of weight $-\dt'+\la$ is spanned by  Chevalley monomials applied to the highest weight vector
$v$. They are products of $f_{\al_1},f_\bt$, and two copies of $f_{\al_i}$, $i=2,\ldots,p$, each.
Every monomial must have the rightmost factor $f_{\al_1}$ as it is the only generator that does not kill $v$.
By Lemma \ref{pre_dt'} all the monomials should have the
factor $f_\bt f_{\al_p}\stackrel{>}{\ldots} f_{\al_1}$ on the right.
We conclude that   vectors of weight $-\dt'+\la$ should be combinations of
$
\phi_\si=g_\si f_\bt f_{\al_p}\stackrel{>}{\ldots} f_{\al_2}f_{\al_1}v
$
with $g_\si=\si(f_{\al_2}\stackrel{<}{\ldots} f_{\al_p})$, where $\si$ is a permutation
of the factors. Suppose $g_\si=\ldots f_{\al_i}f_{\al_k}\stackrel{<}{\ldots} f_{\al_p}$
for some  $k=2,\ldots, p+1$ and $i<k-1$ (we assume formally that
$f_{\al_i}$ is in the rightmost position in $g_\si$  if $k=p+1$). Then
$f_{\al_i}$ can be pushed through to the right of $f_{\bt}$, and this falls under Lemma \ref{pre_dt'}.
Thus, the factors in $g_\si$ are all  ordered as stated, and the permutation
$\si$ is identical.

Finally, the  vector of concern is not zero. Indeed, the
subspace of weight $-\dt'+\la$ in $\hat M_\la$ has the same dimension as the subspace of
weight $-\dt'$ in $U(\n_\l^-)$, which is exactly $1$, due to the Poincar\'e-Birkhoff-Witt (PBW) basis
in $U(\n_\l^-)$.
\end{proof}

Put $\gm=\al_{1}+\ldots +\al_p$, $\dt=2\gm+\bt$,  and  introduce the vector
\be
f_\dt=[f_{\gm},[ \tilde f_{\gm},f_\bt]_{q^{-2}}]_{q^2}
=[ \tilde f_{\gm},[f_{\gm},f_\bt]_{q^2}]_{q^{-2}},
\label{delta_root}
\ee
Remark that the right equality holds by virtue of  Lemma \ref{[tilde,nontilde]}.
Note that $f_\dt$ is not a member of the  standard PBW basis associated with
a reduced decomposition of the longest  Weyl group element, see
Section \ref{SecPBW} for more details.
\begin{lemma}
\label{rearrange_f_delta}
The vector $f_\dt$ is presentable in the form
$$
[f_{\al_1},[f_{\al_{2}}, \ldots [f_{\al_{p}},[f_{\al_1}, \ldots [f_{\al_{p-1}},[f_{\al_p},f_\bt]_{q^2}]_q   \ldots]_{q}]_{q^{-1}}\ldots]_{q^{-1}}]_{q^{-2}}.
$$
\end{lemma}
\begin{proof}
First of all, remark that  $p$ internal commutators amount to $[f_{\gm},f_\bt]_{q^2}$.
Further, fix  $i=2,\ldots, p+1$ and define the root $\nu$ from the equality
$\gm=\nu+\al_{i}+\ldots +\al_{p}$ ($\nu$ is simply $\gm$ if $i=p+1$).
Suppose we have proved that $f_\dt$ is presentable in the form
$
[[f_{\nu},f_{\al_{i}}]_{q^{-1}},z]_{q^{-2}},
$
where $z=[f_{\al_{i+1}}, \ldots [f_{\al_{p}},[f_{\gm},f_\bt]_{q^2}]_{q^{-1}} \ldots]_{q^{-1}}$.
In particular, this is true for $i=p+1$.
The vector $f_\nu$ commutes with everything in $z$ but $f_\gm$.
All the Chevalley generators in $f_\nu$ except from $f_{\al_1}$
commute with $f_\gm$, and $f_{\al_1}$ enters $f_\nu$ exactly once.
 Applying Lemma \ref{[tilde,nontilde]},
we conclude that $[f_\nu,z]_{q^{-1}}=0$.
Using the "Jacobi identity"
\be
[x,[y,z]_a]_b=[[x,y]_c,z]_{\frac{ab}{c}}+c [y,[x,z]_{\frac{b}{c}}]_{\frac{a}{c}},
\label{Jacobi}
\ee
which holds true in any associative algebra with scalar $a,b$, and invertible $c$,
we write
$$
[f_{\nu},[f_{\al_{i}},z]_{q^{-1}},]_{q^{-2}}
=
[[f_{\nu},f_{\al_{i}}]_{q^{-1}},z]_{q^{-2}}
+
q^{-1}[[f_{\nu},f_{\al_{i}}]_{q^{-1}},z],
$$
for $a=c=q^{-1}$, $b=q^{-2}$.
The second term vanishes, and we come to the equality
$
[[f_{\nu},f_{\al_{i}}]_{q^{-1}},z]_{q^{-2}}=[f_{\nu},[f_{\al_{i}},z]_{q^{-1}},]_{q^{-2}}
$. Descending induction on $i=p+1,\ldots, 2$ completes the proof.
\end{proof}

Now we lift the assumption $\ell=0$, $m=1$ and work out the  case of  general
$\k$ and $\l$:
\be
\k&=&\g\l(n_1)\oplus \ldots \oplus\g\l(n_\ell)\oplus \s\p(2m)\oplus\s\p(2p),
\label{stab_subalg}\\
\l&=&\g\l(n_1)\oplus \ldots \oplus\g\l(n_\ell)\oplus \g\l(m)\oplus\s\p(2p).
\label{Levi_subalg}
\ee
Let $\g'=\s\p(2+2p)$ be the Lie subalgebra in $\g$ corresponding to the positive simple roots $(\al_{n-p},\ldots, \al_n)$.
The vectors $f_\gm,\tilde f_\gm,f_\dt \in U_q(\g')$
are carried over to  $U_q(\g)$, where we use the same notation for them.
This
relates the case $\ell=0$, $m=1$ to the general setting.
The root $\al_{n-p}$  plays  the same role as $\al_1$ in the
symmetric case with $m=1$.
We will use the notation $\al$ for it in order to emphasize the general meaning of formulas with it.

\begin{propn}
\label{singular_Verma}
Suppose that  $q^{2(\la,\al)}=-q^{-2p}$. Then $f_\dt v$
is a singular vector in $\hat M_\la$.
\end{propn}
\begin{proof}
At first, we return to the symmetric case $\ell=0$ with $m=1$.
Furthermore, as the case $p=1$ has been studied  in \cite{M3}, we assume $p>1$.

Applying $e_\bt$ to $f_\dt v$ we obtain, up to a non-zero scalar factor,
$$
[f_{\gm},[ \tilde f_{\gm},q^{h_\bt}-q^{-h_\bt}]_{q^{-2}}]_{q^2}
\sim [f_{\gm},[ \tilde f_{\gm},q^{-h_\bt}]_{q^{-2}}]_{q^2}
\sim [f_{\gm}, \tilde f_{\gm}q^{-h_\bt}]_{q^2}=[f_{\gm}, \tilde f_{\gm}q^{-h_\bt}]_{q^2}=
[f_{\gm}, \tilde f_{\gm}]q^{-h_\bt}.
$$
The last commutator is zero, by Lemma \ref{[tilde,nontilde]}.

If $1<i<n-1$, then $e_{\al_i}$ commutes with $f_\gm$, $\tilde f_\gm$, and hence with
$[f_{\gm},[ \tilde f_{\gm},f_\bt]_{q^{-2}}]_{q^2}$, by Lemma \ref{lem_aux1}. Therefore  $e_{\al_i}f_\dt v=0$ for such $i$. Thus, we only need to check that
 $f_\dt v$ is  annihilated by $e_{\al}=e_{\al_1}$ and $e_{\al_p}=e_{\al_{n-1}}$.

Using the formulas
$
[e_{\al_p},f_{\gm}]=-q f_{\gm''}q^{h_{\al_{p}}}
,\quad
[e_{\al_p},\tilde f_{\gm}]=-q^{-1} \tilde f_{\gm''}q^{-h_{\al_{p}}}
$
from Lemma \ref{lem_aux1} we get for $[e_{\al_p},f_\dt]$ an expression, which is proportional to
$$
q [f_{\gm''}q^{h_{\al_{p}}},[ \tilde f_{\gm},f_\bt]_{q^{-2}}]_{q^2}+q^{-1} [f_{\gm},[\tilde f_{\gm''}q^{-h_{\al_{p}}},f_\bt]_{q^{-2}}]_{q^2}.
$$
The second term vanishes, because $[\tilde f_{\gm''}q^{-h_{\al_{p}}},f_\bt]_{q^{-2}}\sim [\tilde f_{\gm''},f_\bt]q^{-h_{\al_{p}}}$=0.
Let us check that the first term is nil too.

For $p=2$, the first term is proportional to
$$
[f_{\al_{p-1}}q^{h_{\al_{p}}},[ \tilde f_{\gm},f_\bt]_{q^{-2}}]_{q^2}
=q[f_{\al_{p-1}},[ \tilde f_{\gm},f_\bt]_{q^{-2}}]_{q}q^{h_{\al_{p}}}
=q[[f_{\al_{p-1}}, \tilde f_{\gm}]_{q},f_\bt]_{q^{-2}}q^{h_{\al_{p}}}=0,
$$
as $[f_{\al_{p-1}}, \tilde f_{\gm}]_{q}=[f_{\al_{1}}, \tilde f_{\gm}]_{q}=0$ by Lemma \ref{[tilde,nontilde]}.

If $p>2$, we present $f_{\dt}$ as
$
f_{\dt}=[ \tilde f_{\nu},[f_{\al_{p-1}},[f_{\al_{p}},[f_{\gm},f_\bt]_{q^2}]_{q^{-1}}]_{q^{-1}}]_{q^{-2}},
$
where  $\nu=\gm -\al_{p-1}-\al_p$ (we did it in the proof of Lemma \ref{rearrange_f_delta}).
This way we unveil $f_{\al_{p-1}}$ hidden in $\tilde f_\gm$.
Although it commutes with $e_{\al_p}$, it does not commute with $q^{\pm h_{\al_p}}$ arising from $[e_{\al_p},f_{\al_p}]$.
Observe that $[e_{\al_p},[f_\gm,f_\bt]_{q^2}]\sim [f_{\gm''}q^{h_{\al_{p}}},f_\bt]_{q^{2}}\sim [f_{\gm''},f_\bt]q^{h_{\al_{p}}}$=0.
Then $[e_{\al_p},f_{\dt}]$ is proportional to
$$
[ \tilde f_{\nu},[f_{\al_{p-1}},[q^{h_{\al_{p}}}-q^{-h_{\al_{p}}},[f_{\gm},f_\bt]_{q^2}]_{q^{-1}}]_{q^{-1}}]_{q^{-2}}
=[ \tilde f_{\nu},[f_{\al_{p-1}},[q^{h_{\al_{p}}},[f_{\gm},f_\bt]_{q^2}]_{q^{-1}}]_{q^{-1}}]_{q^{-2}}=0,
$$
because
$
[f_{\al_{p-1}},[q^{h_{\al_{p}}},[f_{\gm},f_\bt]_{q^2}]_{q^{-1}}]_{q^{-1}}
\sim [[f_{\al_{p-1}},f_{\gm}],f_\bt]_{q^2}q^{h_{\al_{p}}}=0,
$
by Lemma \ref{[tilde,nontilde]}.

We have shown that $f_\dt v$ is annihilated by $e_{\al_i}\in U_q\bigl(\s\p(2p)\bigr)\subset U_q(\g)$.
Next we check that it is killed by $e_{\al_1}=e_\al$.
Based on Lemma \ref{lem_aux1}, we find that $e_{\al} f_\dt v=[e_{\al}, f_\dt] v$ is equal to
$$
[f_{\gm'}q^{-h_{\al}},[ \tilde f_{\gm},f_\bt]_{q^{-2}}]_{q^2}v+[f_{\gm},[\tilde f_{\gm'}q^{h_{\al}},f_\bt]_{q^{-2}}]_{q^2}v=
$$
$$
=q^{1-(\al,\la)}[f_{\gm'},[ \tilde f_{\gm},f_\bt]_{q^{-2}}]_{q}v+q^{(\al,\la)}[f_{\gm},[\tilde f_{\gm'},f_\bt]_{q^{-2}}]_{q}v.
$$
Using Lemma \ref{dt'}, we develop the commutators in $[e_{\al}, f_\dt] v$ and find it  proportional to
$$
(q^{-(\al,\la)-p}+q^{(\al,\la)+p})f_{\al_2}\stackrel{<}{\ldots} f_{\al_p}f_{\bt}f_{\al_p}\stackrel{>}{\ldots} f_{\al_2}f_{\al}v.
$$
It turns zero if and only if $q^{2(\al,\la)}=-q^{-2p}$. This completes the proof for $\ell=0$, $m=1$.

The vector $f_\dt v\in \hat M_\la$ has been shown to be singular with respect to the  subalgebra
$U_q(\g')$, where $\g'=\s\p(2+2p)$ is defined above. Therefore $f_\dt v$ is singular with
respect to entire $U_q(\g)$,
as  $f_\dt\in U_q(\g')$ commutes with simple root vectors $e_\mu\not =U_q(\g')$.
\end{proof}
The following statement presents $f_\dt v$ as a linear combination
of Chevalley monomials.
\begin{propn}
\label{basis_for_f_delta}
The  vector $f_\dt v$ is a linear combination of the monomials
$$
f_{\al_{i}}\stackrel{<}{\ldots} f_{\al_{n-1}} f_\bt f_{\al_{i-1}}\stackrel{>}{\ldots} f_{\al_{n-p}}
f_{\al_{n-1}}\stackrel{>}{\ldots}f_{\al_{n-p}}v,\quad i=n-p+1,\ldots, n.
$$
\end{propn}
\begin{proof}
Follows  from Lemma \ref{pre_dt'} applied to $f_\dt=
f_{\gm} \tilde f_{\gm}f_\bt-{q^{-2}}f_{\gm} f_\bt \tilde f_{\gm} -{q^{2}}\tilde f_{\gm}f_\bt f_{\gm}+f_\bt f_{\gm} \tilde f_{\gm}$.
\end{proof}

From now on, saying the  vector $f_\dt v\in \hat M_\la$ is singular,  we assume that the weight $\la$ satisfies the hypothesis of
Proposition \ref{singular_Verma}.

We  define certain weight subspaces in order to formalize further exposition.
For all $k=1,\ldots, \ell+2$ let $\nu_k$ be an element of the standard orthogonal basis $\{\ve_i\}_{i=1}^n\subset \h^*$ which is not
vanishing on the Cartan subalgebra of the $i$-th block in (\ref{Levi_subalg}), counting from the left.
Put $\mu^0_k=e^{2(\eta,\nu_k)}$ assuming $\eta \in \c^*_\l$.
This definition is independent of the choice of $\nu_k$: for any other $\nu_k$, call it $\nu_k'$, the
 difference $\nu_k-\nu_k'$ is a root of $\l$ and orthogonal to $\eta \in \c^*_\l$. In particular, one can take
 $\nu_{\ell+2}=\ve_n=\frac{1}{2}\bt$. Since $\c^*_\l$ is orthogonal to $\bt$, one has $\mu^0_{\ell+2}=1$.

Let $\c_{\l,reg}^*$ denote the set of all weights $\eta\in \c_\l^*$ such that
$\mu^0_k \not= (\mu^0_j)^{\pm1}$ for $k\not =j$.
Denote by $\c_{\k}^*$ the subset in $\c_\l^*$ such that $\mu^0_{\ell+1}=-1$
and by $\c_{\k,reg}^*$ the subspace of $\c_{\k}^*$ such that $\mu^0_k\not =(\mu^0_j)^{\pm1}$
for $k,j=1,\ldots, \ell+2$ and $k\not =j$. Clearly $\c_{\k,reg}^*$ is dense in $\c_{\k}^*$, being complementary
to a discrete family of hyperplanes. Remark that the
 vector $\mub^0=(\mu^0_i)$ belongs
to $\hat \Mc_K$, which parameterizes the moduli space $\Mc_K$ of classes with fixed $K$.
All elements of $\hat \Mc_K$ can be obtained this way.

The sets $\c^*_{\k}$ and $\c^*_{\k,reg}$ are explicitly described as follows.
Introduce $\ell+2$ weights $\E_i\in \h^*$:
$$\mathcal{E}_1=\ve_1+\ldots +\ve_{n_1}, \quad \mathcal{E}_2=\ve_{n_1+1}+\ldots +\ve_{n_1+n_2},\quad \ldots,\quad\mathcal{E}_{\ell+2}=
\ve_{n-p+1}+\ldots+\ve_n,$$
so that $(\mathcal{E}_i,\nu_k)=\dt_{ik}$.
Then $\c_{\l}^*$ is formed by the combinations $\sum_{i=1}^ {\ell+1} \La_i \E_i$ with arbitrary complex coefficients $\La_i$.
 The subset
$\c_{\k}^*\subset \c_{\l}^*$ is characterized  by the condition $\La_{\ell+1}= \frac{\sqrt{-1}\>\pi}{2}+\sqrt{-1\>}\Z\pi$.
The subsets $\c_{\l,reg}^*\subset \c_{\l}^*$ and $\c_{\k,reg}^*\subset \c_{\k}^*$ are specified  by
$\La_k\not\in  \sqrt{-1}\pi \Z$ and
$\La_k\pm \La_j\not\in  \sqrt{-1}\pi \Z$, for $k,j=1,\ldots,\ell+1$, $k\not =j$.

Finally, we introduce two subsets $\mathfrak{C}^*_{\k,reg}\subset\mathfrak{C}^*_{\k}$ in $\frac{1}{\hbar}\c_\l^*\oplus \c_\l^*$ by
$\mathfrak{C}^*_{\k}=\hbar^{-1} \c_{\k}^*-p\E_{\ell+1}$ and $\mathfrak{C}^*_{\k,reg}=\hbar^{-1} \c_{\k,reg}^*-p\E_{\ell+1}$; obviously $\mathfrak{C}^*_{\k,reg}$ is dense in $\mathfrak{C}^*_{\k}$.
By construction, all weights from $\mathfrak{C}^*_{\k}$ satisfy
$q^{2(\al,\la)}=-q^{-2p}$.

\begin{definition} Assuming $\la \in \mathfrak{C}^*_{\k}$, we denote by $\hat M_{\la-\dt}\subset \hat M_\la$
the submodule generated by $f_\dt v$ and we denote by $M_\la$ the quotient module  $\hat M_\la/\hat M_{\la-\dt}$.
\end{definition}
The module  $M_\la$ is the subject of our further study.

\section{On a basis in $M_\la$}
\label{SecPBW}
In this section we prove that the quotient module $M_\la$ is free over $\C[\![\hbar]\!]$ and construct a PBW basis for it. To that end, we
need a version of the PBW theorem for the nilponent subalgebra $U_\hbar(\g_-)\subset U_\hbar(\g)$. Our consideration is
based on the standard PBW theorem for $U_\hbar(\g)$ with  the advantage of working over $\C[\![\hbar]\!]$,
which makes us more flexible in the choice of basis.

Recall that the quantum version of the higher root vectors in $\g$ can be associated with a reduced decomposition
of the longest element $w_0$ of the Weyl group of $\g$, \cite{ChP}.
Every such decomposition gives rise to an ordered set of root vectors $f_i$, $i=1,\ldots, r=|R^+|$, generating a PBW basis in $U_q(\g_-)$. By construction, higher root vectors are expressed through deformed commutators of the Chevalley generators. By a deformed commutator, we mean $[x,y]_a$, with $a\in \C[\![\hbar]\!]$ satisfying $a=1\mod \hbar$.

We argue that, over $\C[\![\hbar]\!]$, we can a) arbitrarily change the order of the $f_i$'s and b) redefine $f_i$ in an
appropriate although rather general way.
Namely, for $\mu\in R^+$ the root vector $f_\mu$ is a combination $\sum_l a_l \phi_l$,
where $\phi_l$ is a monomial in $\{f_{\al}\}_{\al\in \Pi^+}$ and $a_l \in \C[\![\hbar]\!]$ is a scalar.
Let $\hat f_\mu$ be an element of $U_\hbar(\g_-)$ obtained from $f_\mu$ by replacing
$a_l$  with their deformations $\hat a_l\in \C[\![\hbar]\!]$, i.e. with any $\hat a_l=a_l \mod \hbar$. Note that $a_l$
are themselves deformations of the classical coefficients participating in the
 classical counterpart of $f_\mu$.
 Moreover, if the classical prototype of $f_\mu$ is a combination of commutators
in the Chevalley generators, we can replace them with any deformed commutators. This way  we get rid of
the sequence of  deformed commutators prescribed by the reduced decomposition of $w_0$. Then
$\hat f_\mu$ is still a deformation of $f_\mu$, i.e. $\hat f_\mu=f_\mu\mod \hbar$, and the omitted terms
have the same weight $\mu$.

Consider, for instance, the case of $\s\p(4)$ treated in \cite{M3}. Let $\al,\bt$ be the short and long simple roots, respectively.
There are two more positive roots $\gm=\al+\bt$ and $\dt=2\al+\bt$. The classical root vectors are
$f_\gm=[f_\al,f_\bt]$, $f_\dt=[f_\al,[f_\al,f_\bt]]$, and their standard $q$-counterparts associated with
the decomposition $w_0=s_\bt s_\al s_\bt s_\al$ (simple relfections) are $f_\gm=[f_{\al},f_\bt]_{q^{2}}$, $f_\dt=[f_{\al},[f_{\al},f_\bt]_{q^{2}}]$.
Our version is $\hat f_\gm=[f_{\al},f_\bt]_{q^{2}}$, $\hat f_\dt=[f_{\al},[f_{\al},f_\bt]_{q^{2}}]_{q^{-2}}$, i. e.
the external commutator in $\hat f_\dt$ is deformed as compared to $f_\dt$.
In terms of  Chevalley monomials, the $\dt$-root vectors read
$$
f_\dt=f_\al^2f_\bt-(q^{2}+1)f_\al f_\bt f_\al+q^{2}f_\bt f_\al^2, \quad
\hat f_\dt=f_\al^2f_\bt-(q^2+q^{-2})f_\al f_\bt f_\al+f_\bt f_\al^2.
$$
The coefficients in $\hat f_\dt$ are deformation of the coefficients in
$f_\dt$. In the standard PBW basis  corresponding to the ordering $f_\bt, f_\gm ,f_\dt, f_\al$, one has $\hat f_\dt=f_\dt+(1-q^{-2})f_\gm f_\al= f_\dt \mod \hbar$.

 We fix an arbitrary order on the set $\{\hat  f_i\}$.
 Let $\kb=(k_1,\ldots, k_r)$ denote a multivector with non-negative integer components.
\begin{propn}
The monomials $\hat f_{\kb} =\hat f_1^{k_1}\ldots \hat f_r^{k_r}\subset U_\hbar (\g_-)$ form a basis.
\label{PBW}
\end{propn}
\begin{proof}
First suppose that $(\hat f_i)=(f_i)$ is the standard ordered system of root vectors associated with a reduced decomposition of
$w_0$.
This PBW basis establishes a $\C[\![\hbar]\!]$-linear isomorphism between $U_\hbar(\g_-)$ and $U(\g_-)\tp \C[\![\hbar]\!]$.
It is also an isomorphism of $\h^*$-graded spaces, as it preserves weights.
This isomorphism makes the multiplication of $U_\hbar(\g_-)$ a deformation of the multiplication of $U(\g_-)\tp \C[\![\hbar]\!]$ (the
trivial extension of the ring of scalars), and makes
$U_\hbar(\g_-)$ a deformation of the  $\h^*$-graded algebra  $U(\g_-)$.
Fix a  permutation  $\si$ of $1,\ldots, r$
and define $ g_{\kb} =f_{\si(1)}^{k_{\si(1)}}\ldots  f_{\si(r)}^{k_{\si(r)}}$
to be a PBW monomials in the reordered system of standard root vectors $(f_{\si(i)})$. Let us show that $\{ g_{\kb}\}$
form a basis over $\C[\![\hbar]\!]$.

The linear operator $\Phi\colon  f_{\kb}\mapsto  g_{\kb}=\sum_{\mb}\Phi_{\kb,\mb} f_{\mb}$ preserves the weight subspace, and we must check that it is invertible in every weight subspace. Modulo $\hbar$, $\Phi$ is
relating two PBW bases in the classical universal enveloping algebra $U(\g_-)$.
It is invertible modulo $\hbar$, hence it is invertible in every weight subspace in $U_\hbar(\g_-)$, which is a free
finite $\C[\![\hbar]\!]$-module.

We have proved that  we can arbitrarily change the order of the generators of the standard PBW basis, when working over $\C[\![\hbar]\!]$.
Further, we have $\hat f_i=f_i\mod \hbar $ by construction, and hence $\hat f_{\kb}=f_{\kb}+\hbar \sum_{\mb}\Psi_{\kb\mb}f_{\mb}$
for some $\C[\![\hbar]\!]$-linear operator $\Psi$. Again, we restrict the consideration to the weight subspace
containing $f_{\kb}$. It is free and finite  over $\C[\![\hbar]\!]$, hence the operator $\id +\hbar\Psi$ is
invertible. This completes the proof.
\end{proof}
Further on we drop the symbol $\hat{}$ from $\hat f_i$.
Specifically, we define the higher root vectors as follows.
At the first step, let all $\{f_\mu\}_{\mu\in R^+}$ be the standard root vectors.
Next we redefine the root vectors of every semisimple block of $U_\hbar(\l)$ according to its Weyl group element and
its decomposition.
Finally, we construct a $q$-analog of classical $\m^-=\ad(\l)(f_\dt)$
from $f_\dt$ defined in (\ref{delta_root}) through the operators
$\ad_{op}(f_\mu)\colon U_\hbar(\g)\to U_\hbar(\g)$, $\ad_{op}(f_\mu)\colon y\mapsto f_\mu^{(2)}y \gm^{-1}(f_\mu^{(1)})$,
 where $\mu$ is a simple positive root of
$\g\l(m)\subset \s\p(2m)$. To simplify the enumeration, we restrict to the symmetric case $\l=\g\l(m)\oplus \s\p(2p)$.

All roots from $R_\k^+-R_\l^+$ have the form
$\mu=\sum_{j=i}^{k-1}\al_j+2\sum_{j=k}^{m-1}\al_k+\dt$,
 $1\leqslant i\leqslant k\leqslant m$,
where the left (resp. the right) sum is present only if $i<k$ (resp. $k<m$).
We construct the root vector $f_\mu$  as
$$
f_\mu=\bigl(\ad_{op}(f_{\al_{i}})\circ  \stackrel{<}{\ldots}\circ \>\ad_{op}(f_{\al_{k-1}})\bigr)\circ \bigl(\ad^2_{op}(f_{\al_{k}})\circ
\stackrel{<}{\ldots}\circ \> \ad^2_{op}(f_{\al_{m-1}})\bigr)(f_\dt).
$$
It is a linear combination of Chevalley monomials with scalar coefficients,
since $\ad_{op}(f_{\al_i})x=f_{\al_i} x-q^{-(\al_i,\nu)}xf_{\al_i}$ for any element $x\in U_\hbar(\g)$ of weight $\nu$
and all $\al_i\in \Pi_+$.
 The coefficients
coincide with classical modulo $\hbar$ because the commutators are deformations of the
classical onse. We denote by $\m^-$ the $\C[\![\hbar]\!]$-linear span of $f_\mu$, $\mu\in R_\k^+-R_\l^+$.

Let $\l_-$ denote the linear span of the negative root vectors of $U_\hbar(\l)$.
We regard the module $\hat M_\la$ as that over $U_\hbar(\g_-)$. It is induced from the trivial
representation of the subalgebra $U_\hbar(\l_-)$ and isomorphic to the quotient by the left ideal $U_\hbar(\g_-)\l_-$.
The $U_\hbar(\g_-)$-module $M_\la$ is isomorphic  to the  quotient by left ideal $U_\hbar(\g_-)(f_\dt \oplus \l_-)$.
This ideal is equal to $U_\hbar(\g_-)(\m^-\oplus \l_-)$ since
$f_\dt\subset \m^-$ and  $\m^-\subset \ad_{op}\bigl(U_\hbar(\l_-)\bigr)f_\dt\subset
U_\hbar(\g_-)(f_\dt \oplus \l_-)$.
Now we can establish the main result of this section.
\begin{propn}
\label{M_la is free}
The module $M_\la$ is free over $\C[\![\hbar]\!]$.
\end{propn}
\begin{proof}
The module $M_\la$ is isomorphic to the quotient
of the left regular $U_\hbar(\g_-)$-module by the left ideal $U_\hbar(\g_-)(\m^-\oplus \l_-)$, hence it
is spanned by the BPW monomials with no root vectors from $\m^-\oplus \l_-$. Such monomials
form a basis in $M_\la$.
\end{proof}

\section{Module $\C^{2n}\tp M_\la$: the symmetric case}
\label{secWMsym}
In this section we put $\ell=0$ and work with the Levi subalgebra  $U_q\bigl(\g\l(m)\bigr)\tp U_q\bigl(\s\p(2p)\bigr)$, $m+p=n$.
In this setting, the distinguished root $\al$ is $\al_{n-p}=\al_m$.
It is complementary to the Dynkin sub-diagram of
 $\l$ in the diagram of $\g$.

Consider the natural vector representation of $U_q\bigl(\s\p(2n)\bigr)$ in $\C^{2n}$ and denote
by $\pi$ the homomorphism $U_q\bigl(\s\p(2n)\bigr)\to \End(\C^{2n})$. Let $(w_i)_{i=1}^{2n}\subset \C^{2n}$
be the standard basis, whose elements carry the weights $(\ve_1,\ldots, \ve_n$, $-\ve_n,\ldots,-\ve_1)$.
In this basis, the matrices $\pi(h_\mu)$, $\pi(e_\mu)$ and $\pi(f_\mu)$, $\mu\in \Pi_+$,
are independent of $q$ and are the same as in the classical representation of $U\bigl(\s\p(2n)\bigr)$.

For generic weight $\la \in{\mathfrak{C}^*_{\l,reg}}$, the tensor product $\C^{2n}\tp \hat M_\la$ is the direct sum of three submodules of
highest weights $\nu_1=\ve_1+\la$, $\nu_2=\ve_{m+1}+\la$, $\nu_3=\ve_{n+p+1}+\la$, see. e.g. \cite{M2}.
Let $u_{\nu_i}$, $i=1,2, 3$, denote their generators, which are singular vectors in $\C^{2n}\tp \hat M_\la$.
We are going to prove the direct sum decomposition
$\C^{2n}\tp M_\la=M_1\oplus M_2$, where $M_i$ are the images of $\hat M_i$ under the
projection $\C^{2n}\tp \hat  M_\la\to \C^{2n}\tp M_\la$. The submodule $M_1$ is generated
by $u_{\nu_1}=w_1\tp v$, which is the only singular vector of weight $\nu_1$.
We shall see that the submodule $\hat M_3$ is annihilated under the projection $\C^{2n}\tp\hat M_\la\to \C^{2n}\tp M_\la$
 (in fact,  $u_{\nu_3}$ degenerates to $w_1\tp f_\dt v$ once $\la \in \mathfrak{C}^*_{\k}$
and vanishes in
$\C^{2n}\tp M_\la$, see \cite{M3} for the special case of $\g=\s\p(4)$).

\begin{lemma}
\label{lemma_singular_vec}
The vector
\be
\label{singular_vec}
u_{\nu_2}=\frac{q^{(\al,\la)}-q^{-(\al,\la)}}{q-q^{-1}}w_{m+1}\tp v-q^{-1}w_{m}\tp f_{\al_m} v
+\ldots +(-q)^{-m} w_{1}\tp f_{\al_1}\stackrel{<}{\ldots} f_{\al_m} v
\ee
of weight $\nu_2=\ve_{m+1}+\la$ is singular.
\end{lemma}
\begin{proof}
A straightforward calculation that (\ref{singular_vec}) is annihilated by all $e_\mu$, $\mu\in \Pi^+$.
\end{proof}
Further we develop a diagram technique which will help us study the module $\C^{2n}\tp M_\la$.
Introduce the monomials $\psi_i=f_{\al_i}\stackrel{<}{\ldots} f_{\al_m}\in U_q(\g_-)$, $i=1,\ldots, {m}$,
and write
$$
u_{\nu_2}=\frac{q^{(\al,\la)}-q^{-(\al,\la)}}{q-q^{-1}} w_{m+1}\tp v+ \sum_{i=1}^{m} (-q)^{i-m-1}w_i\tp \psi_iv.
$$
When restricted to the Levi subalgebra $\l=\g\l(m)\oplus \s\p(2p)$, the natural representation of $\g$ on
$\C^{2n}$ splits into three irreducible
sub-representations,
$\C^{2n}=\C^{m}\op \C^{2p}\op \C^{m}$.
The block $\s\p(2p)$ acts on $\C^{2p}$ by the natural representation and trivially on
the other subspaces. The first copy of $\C^m$ supports the natural representation of
$\g\l(m)$, while the second copy of $\C^m$ is the dual representation; the $\g\l(m)$-action  on $\C^{2p}$ is trivial.

The action of $f_{\al_1},\ldots, f_{\al_m}$ on the highest block $\C^{m}\tp M_\la$ can be
conveniently illustrated by the directed  diagram
$$
\begin{array}{ccccccccccc|c}
&&&&&&D_0\\
&\scriptstyle{f_{\al_1}}&&\scriptstyle{f_{\al_2}}&&\scriptstyle{f_{\al_3}}&&\scriptstyle{f_{\al_{m-1}}}&&\scriptstyle{f_{\al_{m}}}&\\\hline
\scriptstyle{w_1\tp \psi_1 v}&\leftarrow& \scriptstyle{w_1\tp \psi_2 }v&\leftarrow& \scriptstyle{w_1\tp \psi_3 v}
&\leftarrow&\scriptstyle{\ldots} &\leftarrow&\scriptstyle{w_1\tp \psi_{m} v}&\leftarrow&\scriptstyle{w_1\tp v}
\\
&&\downarrow&&\downarrow&&  &&\downarrow&&\quad\downarrow &\scriptstyle{f_{\al_1}}\\
&&\scriptstyle{w_2\tp \psi_2 v}&\leftarrow&\scriptstyle{w_2\tp \psi_3 v}&\leftarrow&\scriptstyle{\ldots }&\leftarrow&\scriptstyle{w_2\tp \psi_{m} v}&\leftarrow&\scriptstyle{w_2\tp v}\\
&&&&\downarrow&&  &&\downarrow&&\quad\downarrow &\scriptstyle{f_{\al_2}}\\
&&&&\scriptstyle{w_{3}\tp \psi_{3} v}&\leftarrow&\scriptstyle{\ldots}&\leftarrow&\scriptstyle{w_3\tp \psi_mv}&\leftarrow&\scriptstyle{w_3\tp v}&\\
&&&&&&&&\downarrow&&\quad\downarrow &\scriptstyle{f_{\al_{3}}}\\
&&&&&&\ddots&&\scriptstyle{\vdots }&&\quad \scriptstyle{\vdots }&\scriptstyle{\vdots}\\
&&&&&&&&\downarrow&&\quad\downarrow &\scriptstyle{f_{\al_{m-1}}}\\
&&&&&&&&\scriptstyle{w_{m}\tp \psi_{m}v}&\leftarrow&\scriptstyle{w_{m}\tp v}\\
&&&&&&&&&&\quad\downarrow &\scriptstyle{f_{\al_{m}}}\\
&&&&&&&&&&\scriptstyle{w_{m+1}\tp v}\\
\end{array}
$$
The origin of the diagram is the vertex $w_1\tp v$ in the north-east corner. We call
the north-west to south-east lines diagonals and count them from the origin
down to south-west.
The nodes on the diagram designate the one-dimensional subspaces in $\C^{m}\tp M_\la$  spanned by the
corresponding tensors. The horizontal arrows symbolize the action of the Chevalley generators
on the tensor factor $M_\la$ while the vertical arrows indicate the action on the tensor factor $\C^{m}$.
Each node has two arrows directed from it. The horizonal arrow yields the action on the whole $\C^{m}\tp M_\la$
 (up to an invertible scalar factor) when the associated generator  is distinct from the generator assigned to the vertical
arrow. When both arrows are labeled with the same generator, the latter sends the node to the
two-dimensional space spanned by the nodes down and to the left. Such nodes lie
on the $m$-th diagonal, which is straight above the principal.
The rightmost vertical arrows excepting $f_{\al_m}$ amount to the action on  $\C^{m}\tp M_\la$, as the
 associated  generators kill $v$.

The sub-triangle above the principal diagonal belongs to $M_1$, the submodule
in $\C^{m}\tp M_\la$ generated by $u_{\nu_1}$. That is clear for its rightmost
 column $\{w_i\tp v\}_{i=1}^m$, which nodes are obtained from $w_1\tp v$ by  $\{f_{\al_i}\}_{i=1}^{m-1}\subset\l_-$.
The horizontal and vertical arrows directed from $w_i\tp \psi_jv$, $j>i+1$,
are marked differently. Therefore, every node $w_i\tp \psi_jv$, $j>i$, is proportional to
$f_{\al_j}(w_{i}\tp \psi_{j+1}v)\sim \psi_{j}(w_{i}\tp v)\in M_1$.
Application of $f_{\al_i}$ to $w_{i}\tp \psi_{i+1}v\in M_1$ (which is on the diagonal of the sub-triangle) gives
\be
\label{action_i}
w_{i}\tp \psi_{i}v+ w_{i+1}\tp q^{-1}\psi_{i+1}v \in M_1,\> i<m,\quad
w_{m}\tp \psi_{m}v+ w_{m+1}\tp q^{-(\al,\la)}v\in M_1.
\ee
\begin{lemma}
\label{gl-sym}
The singular vector $u_{\nu_2}$ is equal to $q^{-m}\frac{q^{(\al,\la)+m}-q^{-(\al,\la)-m}}{q-q^{-1}} w_{m+1}\tp v$ modulo $M_1$.
\end{lemma}
\begin{proof}
All nodes above the main diagonal in $D_0$ lie in $M_1$.
Formulas (\ref{action_i}) imply that $ w_{i}\tp \psi_{i}v=-q^{-1}w_{i+1}\tp \psi_{i+1}v $  modulo $M_1$,
for $i\leqslant m$, if we set $\psi_{m+1}=1$. Therefore
$ w_{i}\tp \psi_{i}v=-(-q)^{i-m}q^{-(\al,\la)}w_{m+1}\tp v $ modulo  $M_1$, for $i=1,\ldots, m$.
Then (\ref{singular_vec}) gives
$$
u_{\nu_2}=\Bigl(\frac{q^{(\al,\la)}-q^{-(\al,\la)}}{q-q^{-1}} + q^{-(\al,\la)-1}\sum_{i=0}^{m-1} q^{2(i-m)}\Bigr) w_{m+1}\tp v
\mod M_1.
$$
Now the proof is immediate.
\end{proof}
In order to prove the direct decomposition $\C^{2n}\tp M_\la=M_1\op M_2$,
we develop our diagram technique  further.
Introduce  monomials $\phi_i\in U_q(\g_-)$, $i=1,\ldots, p+1,$ of degree $2p+1$ by the formulas (recall that $m=n-p$ in this section)
$$
\phi_{i}:=(f_{\al_{m+i-1}}\stackrel{<}{\ldots} f_{\al_{n-1}} f_\bt f_{\al_{m+i-2}}\stackrel{>}{\ldots} f_{\al_m})(f_{\al_{n-1}}\stackrel{>}{\ldots} f_{\al_m}),\quad i=1,\ldots, p+1.
$$
According to Proposition \ref{basis_for_f_delta}, the root vector $f_\dt$ is a linear combination of $\phi_i$.

Denote by $f_i^l$, $l=1,\ldots,2p+1$, the $l$-th factor in $\phi_i$ counting from the right.
By construction, $f_i^l =f_{\al_{l+m-1}}$ for  $1\leqslant l\leqslant p$ and all $i$.
The other $p=n-m$ elements $f_i^l$ for $l=p+1,\ldots, 2p+1$ are obtained by
a permutation of the leftmost $p$ terms, including $f_\bt$, of  the sequence
\be
f_{\al_m},\ldots, f_{\al_{n-1}}, f_\bt,
f_{\al_{n-1}} ,\ldots, f_{\al_m}.
\label{ordered1}
\ee
Denote by $\phi_i^l$ the product $f_i^l\stackrel{>}{\ldots}  f_i^1$ for all $l=1,\ldots,2p+1$.
In particular, $\phi_i^{l}=f_{\al_{l-1+m}}\stackrel{>}{\ldots}  f_{\al_m}$ for all $1\leqslant l\leqslant p$, and
$\phi_i^{2p+1}=\phi_i$. It is also convenient to put $\phi_i^{0}:=1$ for all $i$.

With every $i=1,\ldots, p+1,$ we associate a diagram $D_i$ of $p+1$ rows if $i>1$ and of $2p+2$ rows if $i=1$.
The lengths of the rows vary from  $2p+2$ to $1$ in $D_1$ and to $p+2$ in $D_i$, $i>1$, from top to bottom. The rows
are leveled on the right, so $D_1$ is a  triangle and $D_i$ is a trapezoid for $i>1$. All $D_i$ can
be extended further down  as $D_1$, but we need only their first $p+1$ rows.

The rightmost column in $D_i$ is formed by the tensors $w_{m+l-1}\tp v$, where $l$ runs from $1$ to
$p+1$ if $i>1$ and to $2p+2$ in $D_1$. The intersection of  $l$-th row and $j$-th
column is the tensor $ w_{m+l-1}\tp \phi_i^{j-1}v$. As before, the nodes span one-dimensional subspaces
in $ \C^{2n}\tp M_\la$ and the arrows designate the action of $f_{\al_i}$: horizontal on $M_\la$ and
vertical on $\C^{2n}$.
In all diagrams the vertical
arrow applied to the $j$-th   row is labeled with $f_{1}^j$,
i.e. the $j$-th  term in (\ref{ordered1}) from the right. The horizontal arrows
are the factors $f_i^l$ constituting $\phi_i$.

If the generators assigned to the two arrows directed from  a node are distinct, the horizontal arrow gives the
action on $\C^{2n}\tp M_\la$, up to an invertible scalar factor. If they coincide, the node is sent to the span of the two nodes: next down and next to the left. Modulo the down node, the horizontal arrow still gives the action on
$\C^{2n}\tp M_\la$, up to an invertible scalar. This follows from the coproduct $\Delta(f_\mu)=f_\mu\tp q^{-h_\mu}+1\tp f_\mu$, $\mu\in \Pi^+$.
$$
\scriptstyle{
\begin{array}{ccccccccc|ccc}
&&&&D_1\\
&\scriptstyle{f_1^{2p+1}}&&\scriptstyle{f_1^{2p}}&\ldots &\scriptstyle{f_1^{2}}&&\scriptstyle{f_1^{1}}&\\\hline
\scriptstyle{w_m\tp \phi_1^{2p+1} v}&\leftarrow&\scriptstyle{w_m\tp \phi_1^{2p} v}&\leftarrow &\ldots&\leftarrow&\scriptstyle{w_m\tp \phi_1^{1} v}&\leftarrow &
\scriptstyle{w_m\tp  v}&\\
&&\downarrow&&&&\downarrow&&\downarrow &\scriptstyle{f_1^{1}}\\
&&\scriptstyle{w_{m+1}\tp \phi_1^{2p} v}&\leftarrow&\ldots&\leftarrow &\scriptstyle{w_{m+1}\tp  \phi_1^{1}v}&\leftarrow &\scriptstyle{w_{m+1}\tp  v}&\\
&&&&&&\downarrow&&\downarrow&\scriptstyle{f_1^{2}}\\
&&&&\ddots&&\vdots&&\vdots&\vdots\\
&&&&&&\downarrow&&\downarrow&\scriptstyle{ f_1^{2p}}\\
&&&&&&\scriptstyle{ w_{n+p}\tp \phi_1^1 v}&\leftarrow&\scriptstyle{ w_{n+p}\tp v}\\
&&&&&&&&\downarrow &\scriptstyle{f_1^{2p+1}}\\
&&&&&&&&\scriptstyle{ w_{n+p+1}\tp v}\\
\end{array}
}
$$
$$
\scriptstyle{
\begin{array}{cccccccccc|cc}
&&&&D_i, \quad i> 1\\
&\scriptstyle{f_i^{2p+1}}&&\scriptstyle{f_i^{2p}}& \ldots & &\scriptstyle{f_i^{p+1}}&\ldots&\scriptstyle{f_i^{1}}&\\\hline
\scriptstyle{w_m\tp \phi_i^{2p+1} v}&\leftarrow&\scriptstyle{w_m\tp \phi_i^{2p} v}&\leftarrow&\ldots&
\scriptstyle{ w_{m}\tp \phi_i^{p+1} v}&\leftarrow&\ldots&\leftarrow&\scriptstyle{w_m\tp v}&\\
&&\downarrow&&&\downarrow&&&&\downarrow&\scriptstyle{f_1^{1}}\\
&&\scriptstyle{w_{m+1}\tp \phi_i^{2p} v}&\leftarrow&&&&&&&\\
&&&&\ddots&\vdots&&&&\vdots&\vdots\\
&&&&&\downarrow&&&&\downarrow&\scriptstyle{f_1^{p}}\\
&&&&&\scriptstyle{ w_{m+p}\tp \phi_i^{p+1} v}&\leftarrow&\ldots &\leftarrow&
\scriptstyle{ w_{m+p}\tp v}&\\
\end{array}
}
$$

We present the diagrams $D_1,D_2,D_3$ in Appendix, in order to illustrate the formalism in the case of $m=1$, $p=2$, $n=3$.

In the following theorem we regard $U_q(\g)$ as a $\C$-algebra generated by $e_\mu,f_\mu,t_\mu$, $\mu\in \Pi^+$, assuming
$q^4\not =1$.
\begin{thm}
\label{sp-sym}
Suppose that  $q^{-2p+2m}\not =-1$. Then the  $U_q(\g)$-module  $\C^{2n}\tp M_\la$ is isomorphic to the direct sum $M_1\op M_2$.
\end{thm}
\begin{proof}
There exists a $U_q(\g)$-invariant operator $\Q=(\pi\tp \id)(\Ru_{21}\Ru)$ on $\C^{2n}\tp U_\hbar(\g)$,
where $\Ru$ is the universal R-matrix of $U_\hbar(\g)$. This operator plays an important role in this exposition
and is discussed at length in Section \ref{secQCC}. Here we need the information on its eigenvalues, which
 are found in \cite{M2}. Specifically on $\hat M_1$ and $\hat M_2$, the operator $\Q$ is a scalar multiplier
by, respectively, $q^{2(\la,\ve_1)}$ and
$q^{2(\la,\ve_{m+1})-2m}$. Hence its  eigenvalues on $M_1$ and $M_2$ are $-q^{-2p}$ and $q^{-2m}$,
assuming $\la \in \mathfrak{C}^*_{\k}$. Under the hypothesis of the theorem, they are distinct, therefore
the submodules $M_1$ and $M_2$ have  zero intersection.

We must show that the sum $M=M_1\op M_2$ exhausts all of $\C^{2n}\tp M_\la$. It is sufficient
 to prove that $\C^{2n}\tp v$ lies in  $M$.
Since $v$  is annihilated by $f_{\al_i}$, $i=1,\ldots,m-1$,  we have $f_{\al_i}(w_i\tp v)=w_{i+1}\tp v$. Therefore, $w_i\tp v   \in M_1$
for $i=1,\ldots, m$.
By Lemma \ref{gl-sym}, the vector $w_{m+1}\tp v$  belongs to $M$ if $q^{2(\al,\la)+2m}=-q^{-2p+2m}\not =1$.
The Chevalley generators $f_{\al_{i}}$, $i=m+1, \ldots, n$, belong to the Levi subalgebra and kill $v$.
Applying them repeatedly to $w_{m+1}\tp v\in M$ we get $w_{l}\tp v$ for all $l=m+2,\ldots, n+p$ and
 prove that they are in $M$.

The crucial step is  to show that $w_{n+p+1}\tp v\in M$.
First of all, the triangle above the principal diagonal of $D_1$ lies in $M$. This is checked by induction on the column number.
Let $C_l$ be the linear span of the nodes of column $l$ above
the principal diagonal and $C_l'$ the linear span of nodes from $C_l$ without the bottom one.
We have proved that $C_1$ lies in $M$. Suppose it is true for some column $l\geqslant 1$.  Let $f_1^l$ be the Chevalley generator assigned to the horizontal arrow from column $l$ to column
$l+1$. It sends $C'_l$ isomorphically to $C_{l+1}$ modulo $C_l$, which lies in $M$ by the induction assumption.
Therefore $C_{l+1}$ lies in  $M$.
The left diagram below displays schematically the induction transition.
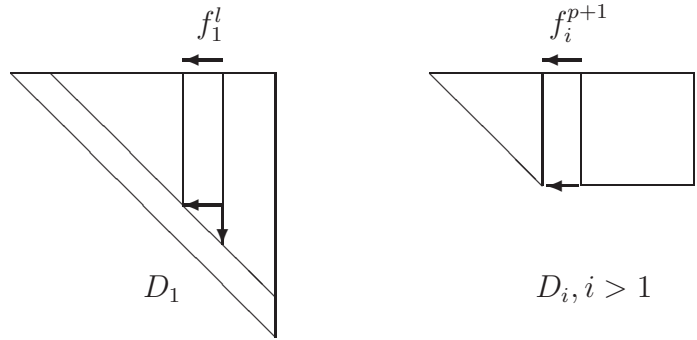
\begin{center}
\begin{picture}(100,125)
\put(50,15){$D_1$}
\put(100,0 ){\line(0,1){100}}\put(0,100 ){\line(1,0){100}}\put(0,100 ){\line(1,-1){100}}
\thinlines\put(15,100 ){\line(1,-1){85}}
\put(80,100 ){\line(0,-1){65}}\put(65,100 ){\line(0,-1){50}}
\thicklines\put(80,50){\vector(0,-1){15}}\thicklines\put(80,50){\vector(-1,0){15}}
\thicklines\put(80,105){\vector(-1,0){15}}
\put(70,115){$f^l_1$}
\end{picture}
\hspace{0.7in}
\begin{picture}(100,125)
\put(40,15){$D_i,i>1$}
\put(100,100 ){\line(0,-1){42.5}}\put(0,100 ){\line(1,0){100}}\put(0,100 ){\line(1,-1){42.5}}\put(57.5,57.5 ){\line(1,0){42.5}}
\put(42.5,100 ){\line(0,-1){42.5}}\put(57.,100 ){\line(0,-1){42.5}}
\thicklines\put(56.,57.5){\vector(-1,0){12}}
\thicklines\put(57.5,105){\vector(-1,0){15}}
\put(45,115){$f_i^{p+1}$}
\end{picture}
\end{center}

Now we are going to prove that the nodes
$w_{m+l}\tp \phi_1^{2p+1-l}v$, $l=0,\ldots ,2p+1$, on the main diagonal
belong to $M$. To compare vectors modulo $M$ we will use the  symbol $\equiv$, i.e.
$x\equiv y $ if and only if $x-y\in M$.
Consider the diagonal next to the main. We already know that its nodes are in $M$.
The Chevalley generators assigned
to the horizontal and vertical arrows coincide at every node on this diagonal, hence each of
them is mapped to a linear combination of two nodes on the principal diagonal. Their images are
$$
a_1 w_m\tp \phi_1^{2p+1}v+ w_{m+1}\tp \phi_1^{2p}v\equiv 0,\quad\ldots\quad,
a_{2p+1} w_{n+p}\tp \phi_1^1v+ w_{n+p+1}\tp\phi_1^{0}v\equiv 0,
$$
where $a_i$ are non-zero scalars. Thus, all nodes on the principal diagonal of $D_1$ are proportional to each other,
modulo $M$.

Now we turn to the diagram $D_i$, $i=2,\ldots, p+1$. Observe that  its first $p$ columns on the right
coincide with the corresponding square part of $D_1$ (the first $p$ horizontal arrows are the same in all
diagrams). That part is situated above the principal diagonal in $D_1$ and
therefore lies in $M$. Since $f_i^{p+1}w_{n+p}=0$ for $i>1$, the operator $f_i^{p+1}$ is mapping
the node $w_{m+p}\tp \phi_i^{p} v$ onto $w_{m+p}\tp \phi_i^{p+1} v$.
 This implies that the column $p+1$ in $D_i$ lies in  $M$.
 A simple induction proves that the leftmost triangular part of $D_i$ including column $p+1$ lies in $M$.
The reasoning is similar to what we did for the triangle in  $D_1$ above the main diagonal.
 In particular, the tensor  $w_m\tp \phi_i^{2p+1}=w_m\tp \phi_i$, $i>1$, belongs to  $M$.

Now recall from Proposition \ref{basis_for_f_delta} that $f_\dt$ is a linear combination of $\phi_i$. Adding the equalities
$ w_{m}\tp \phi_i^{2p+1}v\equiv 0$ for
$i=2,\ldots, 2p+1$, with appropriate multipliers, to the equality $a_1 w_m\tp \phi_1^{2p+1}v+ w_{m+1}\tp \phi_1^{2p}v\equiv 0$
 we replace
$ w_m\tp \phi_1^{2p+1}v$ with $ w_m\tp f_\dt v $, which is nil in $\C^{2n}\tp M_\la$.
This gives $w_{m+1}\tp \phi_1^{2p}v\equiv 0$. On the other hand, we  proved that all
the nodes on the main diagonal of $D_1$ are proportional modulo $M$.
Moving down the diagonal  we eventually conclude that $w_{n+p+1}\tp v\equiv 0$.

To complete the proof, we must check that $w_{j}\tp v \in M$ for $j>{n+p+1}$.
These tensors belong to $\C^m\tp M_\la$, where $\C^m\subset \C^{2n}$ is the irreducible module of
the Levi subalgebra generated by $w_{n+p+1}$. We have  $M\supset U_\hbar(\l)(w_{n+p+1}\tp v)=\C^m\tp v$.
Thus,  $\C^{2n}\tp v$ is contained in $M$, and  $M=\C^{2n}\tp M_\la$.
\end{proof}
The direct sum decomposition is a strong property, which is hard to prove for general $\k$.
 For our purposes, it is sufficient to replace it with an increasing filtration,
 which construction is easier. We rephrase Theorem \ref{sp-sym} for the symmetric case in this
 milder setting, which will be a part of a construction for general $\k$ further on.

Set $V_1=M_1$ to be the $U_q(\g)$-module generated by $w_1\tp v$ and denote by $V_2$ the $U_q(\g)$-module generated by
$\{w_1\tp v, w_{m+1}\tp v\}$, so that
 $V_1\subset V_2$. While $w_{m+1}\tp v$ is not a singular vector, it is so modulo
$V_1$. Identified with its projection to $V_2/V_1\simeq M_2$, it is a highest weight vector
in the quotient $V_2/V_1$.
\begin{propn}
\label{sym_filt}
The module $V_2$ coincides with $\C^{2n}\tp M_\la$, and $V_2/V_1\simeq M_2$.
\end{propn}
\begin{proof}
Using a similar reasoning as in the proof of Theorem \ref{sp-sym},
we show  that $\C^{2n}\tp v$ and hence $\C^{2n}\tp M_\la$ lie in $V_2$. The only difference is that the
inclusion $w_{m+1}\tp v\subset V_2$ holds by the very construction, and this
is a simplification.
\end{proof}
\section{Module $\C^{2n}\tp M_\la$: general case}
For the general Levi subalgebra $\l$, the vector space $\C^{2n}$ decomposes in the direct sum of irreducible $\l$-submodules,
\be
\C^{2n}=W_1\op \ldots \op W_{\ell+1}\op W_{\ell+2}\op  W_{\ell+3}\op\ldots \op W_{2\ell+3},
\nn
\ee
 of dimensions $n_1,\ldots, n_\ell,m,2p,m,n_\ell,\ldots,n_1$. This decomposition corresponds to the block-diagonal structure
of $\l$. For $i=1,\ldots, \ell+1$, the block $\g\l(n_i)$ acts by the natural representation on
 $W_i\simeq \C^{n_i}$, by the dual representation on $W_{2\ell+4-i}\simeq \C^{n_i}$, and trivially
on the other spaces. The bock  $\s\p(2p)$ acts by the natural representation on $W_{\ell+2}\simeq\C^{2p}$
and trivially on the other subspaces.
The highest weights $\nu_i$ of $W_i$ are
\be
\ve_{1},\>\ve_{n_1+1},\ldots, \ve_{n_1+\ldots +n_\ell+1},\>
\ve_{n_1+\ldots + n_\ell+m+1}, \>-\ve_{n_1+\ldots+ n_\ell+m},\>-\ve_{n_1+\ldots +n_\ell},\ldots,
-\ve_{n_1}.
\label{hwC2n}
\ee
The highest weight vectors $w_{\nu_i}$, $i=1,\ldots ,2\ell+3$, belong to the standard basis $\{w_i\}_{i=1}^{2n}\subset\C^{2n}$.

For generic weight $\la \in \c_{\l,reg}^*$  the decomposition of $\C^{2n}$ induces the decomposition
\be
\C^{2n}\tp  \hat M_\la=\oplus_{i=1}^{2\ell+3}\hat  M_i
\label{U_ql-decomp}
\ee
of $U_q(\g)$-submodules.
The blocks are generated by singular vectors of weights $\nu_i+\la$, where $\nu_i$ are given by (\ref{hwC2n}).
This follows from non-degeneracy of the contravariant form on  $\hat M_\la$ and  $\hat  M_i$ at generic $\la$,
see e.g. \cite{M2}.

The $\l$-modules
$W_{\ell+1}$ and $W_{\ell+3}$ are merged into a single irreducible $\k$-module that  supports
the natural representation of the block $\s\p(2m)\subset \k$. The other  $\l$-submodules in $\C^{2n}$ remain irreducible with respect
to $\k$.

Denote by $M_i$ the images of  $\hat M_i$ under the projection $\C^{2n}\tp \hat M_\la\to \C^{2n}\tp M_\la$.
One should expect that $\hat M_{\ell+3}$ is annihilated by the projection, and
 decomposition (\ref{U_ql-decomp}) turns into
$$\C^{2n}\tp  M_\la=M_1\op \ldots \op M_{\ell+1}\op M_{\ell+2}\op M_{\ell+4}\op\ldots \op M_{2\ell+3}.$$
However,  this is not easy to prove in the general case. On the other hand, all we need is
the spectrum of the invariant operator $\Q\in \End(\C^{2n}\tp M_\la)$, cf. Section \ref{secQCC}. It
  is sufficient for that to replace the direct sum with
a suitable filtration, which is easier.

Denote by $ V_k$ the $U_q(\g)$-submodule in  $ \C^{2n}\tp  M_\la$ generated by
$\{w_{\nu_i}\tp v\}_{i=1,\ldots, k}$ assuming $k=1,\ldots, 2\ell+3$. We have the obvious inclusion
$
 V_{k-1}\subset  V_{k}.
$
It is convenient to set $ V_0=\{0\}$.
\begin{propn}\label{filt_WM}
The $U_q(\g)$-modules $\{0\}=V_0\subset V_1\subset \ldots \subset V_{2\ell+3}= \C^{2n}\tp  M_\la $ form an ascending filtration.
For each $k=1,\ldots, 2\ell+3$ the graded component $V_{k}/V_{k-1}$ is either $\{0\}$ or
a highest weight module generated by (the image of) $w_{\nu_k}\tp v$.
In particular, $V_{\ell+3}/V_{\ell+2}=\{0\}$.
\end{propn}
\begin{proof}
Observe that $e_{\mu}(w_{\nu_k}\tp v)=0 \mod V_{k-1}$ for all $\mu\in \Pi^+$, i.e. $w_{\nu_k}\tp v$ is a singular vector in $V_k/V_{k-1}$
unless it is nil.
Since $V_k/V_{k-1}$ is generated by $w_{\nu_k}\tp v$, it is the highest weight vector unless
$V_k/V_{k-1}=\{0\}$.

We will prove the inclusion  $\oplus_{i=1}^k W_i\tp v\subset V_k$ for $k=2\ell+3$
and $\C^{2n}\tp v\subset V_{2\ell+3}$ in particular.
  This will imply
$V_{2\ell+3}= \C^{2n}\tp  M_\la$.

 Suppose we have proved that $W_k\tp v\subset V_k$ for some $k> 0$. This is also
true for $k=0$ if we set $W_0=\{0\}$.
By construction, $w_{\nu_{k+1}}\tp v\in V_{k+1}$.
Then $U_q(\l_-)(w_{\nu_{k+1}}\tp v)=U_q(\l_-)w_{\nu_{k+1}}\tp v=W_{k+1}\tp v\subset V_{k+1}$ because
$\l_-$ annihilates $v$ and $w_{\nu_{k+1}}$ generates the $U_q(\l)$-submodule $W_{k+1}$.
 Induction on $k$ proves $W_k\tp v\subset V_k$ and therefore $\oplus_{i=1}^k W_i\tp v\subset V_k$
for all $k$.

Finally, let us prove the equality $V_{\ell+2}= V_{\ell+3}$. It is sufficient to check the inclusion
$w_{\nu_{\ell+3}}\tp v\in V_{\ell+2}$,
because $V_{\ell+3}/V_{\ell+2}$ is generated by $w_{\nu_{\ell+3}}\tp v$.
This boils down to the symmetric case studied in Proposition \ref{sym_filt}. The vector
$w_{n-p}$ belongs to $W_{\ell+1}$, so $w_{n-p}\tp v\in V_{\ell+2}$. Also, $w_{n-p+1}\tp v\in V_{\ell+2}$, because $w_{n-p+1}=w_{\nu_{\ell+2}}$.
Consider the Lie subalgebra $\g'=\s\p(2+2p)\subset \s\p(2n)$ defined in Section \ref{secGVMM}.
Let $M_\la'\subset M_\la$ be the $U_q(\g')$-submodule in $M_\la$ generated by $v$. Consider the natural representation
of $U_q(\g')$
on $\C^{2+2p}$ with the highest weight vector $w_{n-p}$ and the lowest weight vector $w_{n+p+1}=w_{\nu_{\ell+3}}$.
Let $V'_2$ be the
$U_q(\g')$-submodule generated by  $w_{n-p}\tp v$ and $w_{n-p+1}\tp v$. By Proposition \ref{sym_filt},
$V_2'=\C^{2+2p}\tp M_\la'$, hence $w_{n+p+1}\tp v\in V_2'\subset V_{\ell+2}$, as required.
 This completes the proof.
\end{proof}

\section{The matrix of quantum coordinate functions}
\label{secQCC}
The classical description of semi-simple conjugacy classes is formulated in terms of operations
(multiplication, transposition, trace functional) with the matrix $A$ of
 coordinate functions on $\End(\C^{2n})$. The matrix  $A$ is $G$-invariant, and its entries
generate the polynomial algebra of the class.
 A similar description of the quantum conjugacy classes involves a matrix $A$ with non-commutative
 entries or its image $\Q\in  \End(\C^{2n})\tp U_q(\g)$, which should be regarded as the "restriction"
 of $A$ to the  "quantum group" $G_q$.  In this section, we study algebraic properties of $\Q$.

The operator $\Q$ is defined  through the universal  R-matrix of $U_\hbar(\g)$,
which is an invertible element of (completed) tensor square of  $U_\hbar(\g)$, conventionally
 denoted by $\Ru$:
$$
\Q=(\pi\tp\id)(\Ru_{12}\Ru)\in \End(\C^{2n})\tp U_q(\g).
$$
That the entries of $\Q$ lie in $U_q(\g)\subset U_\hbar(\g)$ follows from the explicit expression
of the universal R-matrix, see \cite{ChP}.
Regarded as an operator on $\C^{2n}\tp \hat M_\la$, it satisfies a polynomial equation with the roots \cite{M2}
$$
q^{2(\la+\rho,\nu_i)-2(\rho,\ve_1)}=q^{2(\la,\nu_i)+2(\rho,\nu_i-\ve_1)},
$$
where $\{\nu_i\}_{i=1}^{2\ell+3}$ are the highest weights of the irreducible $\l$-submodules $W_i\subset\C^{2n}$ and $\rho$ the half-sum of the
positive roots of $\g$.

Assuming $\la \in \mathfrak{C}_{\k,reg}^*$, put $\La_i=\hbar(\la,\ve_{n_1+\ldots+n_{i-1}+1})=(\la,\ve_{n_1+\ldots+n_{i}})\in \C$
for $i=1,\ldots, \ell$.
Define $\mub\in \C^{\ell+2}[q,q^{-1}]$ by
\be
\mu_i=e^{2\La_i}q^{-2({n_1+\ldots+n_{i-1}})},
\quad
\mu_{\ell+1}=-q^{-2({n_1+\ldots+n_{\ell}}+p)},
\quad
\mu_{\ell+2}=q^{-2({n_1+\ldots+n_{\ell}}+m)},
\label{mu_param}
\ee
where $i=1,\ldots, \ell$.
The eigenvalues of $\Q$ on $\C^{2n}\tp \hat M_\la$  are expressed through $\mub$ by
\be
\label{e_v in hatM_la}
\mu_i,\quad \mu_i^{-1}q^{-4n+2(n_i-1)},\quad i=1,\ldots ,\ell+1,\quad\mbox{and}\quad\mu_{\ell+2}.
\ee
Recall that, for $\la\in \mathfrak{C}_{\k,reg}^*$,  the classical limit
$\mub^0=\lim_{q\to1}\mub \subset \hat \Mc_K$ parameterizes the moduli space  of classes with the stabilizer $K$.
Explicitly, $\mu_i^0=e^{2\La_i}$ for $i=1,\ldots,\ell$, $\mu_{\ell+1}^0=-1$, and $\mu_{\ell+2}^0=1$.

\begin{propn}
\label{propn_evQM}
For $\la\in \mathfrak{C}_{\k,reg}^*$, the operator  $\Q\in \End(\C^{2n}\tp M_\la)$ satisfies a
polynomial equation of degree $2\ell+2$  with the roots
\be
\mu_i,\quad \mu_i^{-1}q^{-4n+2(n_i-1)},\quad i=1,\ldots, \ell,\quad\mbox{and}\quad\mu_{\ell+1},\quad\mu_{\ell+2}.
\label{e_v in M_la}
\ee
\end{propn}
\begin{proof}
The proof is based on the following fact: a linear operator on a complex vector space
 is semi-simple if and only if it satisfies a polynomial equation with simple roots.
This is true for finite dimensional vector spaces, but in our case we can restrict
the consideration to every weight subspace in $\C^{2n}\tp \hat M_\la$, which is finite dimensional
and $\Q$-invariant.

It is known that $\Q$, as an operator on $\C^{2n}\tp \hat M_\la$, satisfies  a polynomial equation
of degree $2\ell+3$ with the roots (\ref{e_v in hatM_la}), cf. \cite{M2}. Its eigenvalues
are pairwise distinct in  the classical limit, apart from
$\lim_{q\to 1}\mu_{\ell+1}=\lim_{q\to 1}\mu_{\ell+1}^{-1}q^{-4n+2(m-1)}=-1$.
However, for $q\not =1$ this coincidence is no longer the case, and the eigenvalues (\ref{mu_param}) become pairwise distinct
for $q$ in a punctured neighborhood of $1$: "quantization eliminates degeneration". This implies that
$\Q$ is semi-simple on $\C^{2n}\tp \hat M_\la$ for all $q$ close to $1$ and hence for generic $q$
(for fixed $\la$, the entries of $\Q$ in every weight subspace are polynomials in $q^{\pm1}$).
Therefore it is semi-simple on the quotient $\C^{2n}\tp  M_\la$, where $\mu_{\ell+1}^{-1}q^{-4n+2(m-1)}$
is no longer its eigenvalue, by Proposition \ref{filt_WM}.
This proves the statement    for generic $q$ and therefore for all $q$.
\end{proof}
The matrix $\Q$ produces  central elements of $U_\hbar(\g)$ via the q-trace construction.
Since $\Q$ commutes with $\Delta U_\hbar(\g)$, the elements
\be
\label{q-trace}
\Tr_{q}(\Q^k):=\Tr\bigl((\pi(q^{2h_\rho})\tp 1)\Q^k\bigr)\in U_\hbar(\g), \quad k=1,2,\ldots
\ee
are invariant under the adjoint action $\ad(u)x=u^{(1)}x \gm(u^{(2)})$, $u,x\in U_\hbar(\g)$.
It is a standard fact from the Hopf algebra  theory that $ad$-invariant elements are central and {\em vice versa}.
We will use the shortcut notation $\tau_k$ for $\Tr_q({\Q^k})$, $ k=1,2,\ldots$.

A  $U_\hbar (\g)$-module of highest weight $\la$ defines a one-dimensional representation $\chi^\la$ of the centre of  $U_\hbar (\g)$,
which assigns a scalar to each  $\tau_k$:
\be
\chi^\la(\tau_k)=
\sum_{\nu} q^{2k(\la+\rho,\nu)-2k(\rho,\nu_1)}
\prod_{\al\in \Rm_+}\frac{ q^{(\la+\nu+\rho,\al)}-q^{-(\la+\nu+\rho,\al)}}{ q^{(\la+\rho,\al)}-q^{-(\la+\rho,\al)}}.
\label{char_V}
\ee
The summation is taken over the weights $\nu\in \{\pm \ve_j\}_{j=1}^n$ of the module $\C^{2n}$.
Restriction of $\la$ to $\mathfrak{C}_{\k,reg}^*$ makes
the right-hand side a function of the vector $\mub$ defined in (\ref{mu_param}). We denote this function by
$\vt_{\nb,q}^k(\mub)$, where $\nb=(n_1,\ldots,n_\ell, m,p)$ is the integer valued vector of multiplicities.
In the limit $\hbar \to 0$ the function $\vt_{\nb,q}^k(\mub)$ goes over
into the right-hand side of (\ref{tr_cl}), where
$\mub$ should be replaced with $\mub^0=\lim_{q\to 1}\mub$.

\section{Quantum conjugacy classes of non-Levi type}
By quantization of a commutative $\C$-algebra $\A$ we understand a
$\C[\![\hbar]\!]$-algebra $\A_\hbar$, which is free as a $\C[\![\hbar]\!]$-module,
and $\A_\hbar/\hbar\A_\hbar\simeq \A$ as a $\C$-algebra. Note that we
do not require $\hbar$-adic completion because the algebras of our interest
are  direct sums of $U_\hbar(\g)$-submodules,
which we prefer to preserve  under quantization. Below we describe
the quantization of $\C[G]$ along the Poisson bracket (\ref{poisson_br_sts}).

Recall (see e.g. \cite{FRT}) that the image of the universal R-matrix  in the
natural representation is equal, up to a scalar factor, to
$$
R=\sum_{i,j=1 }^{2n} q^{\delta_{ij}-\delta_{ij'}}e_{ii}\tp e_{jj}
  +
  (q-q^{-1})\sum_{i,j=1 \atop i>j}^{2n}(e_{ij}\tp e_{ji}
- q^{\rho_i-\rho_j}\epsilon_i\epsilon_j
e_{ij}\tp e_{i'j'}),
$$
where $\rho_i=-\rho_{i'}=(\rho,\ve_i)=n+1-i$ for $i=1,\ldots, n$.

Denote by $S$ the operator $PR\in \End(\C^{2n})\tp \End(\C^{2n})$,
where $P$ is the ordinary flip of $\C^{2n}\tp \C^{2n}$.
It commutes with the $U_\hbar(\g)$-action on $\C^{2n}\tp \C^{2n}$
and generates three invariant idempotents. One of them
is a one-dimensional projector $\kappa$ onto the
 trivial  $U_\hbar(\g)$-submodule; it is proportional to
$\sum_{i,j=1}^{2n}q^{\rho_i-\rho_j}\epsilon_i\epsilon_j e_{i'j}\tp e_{ij'}.
$

Denote by  $\C_\hbar[G]$ the associative algebra generated by
the entries of a matrix $A=||A_{ij}||_{i,j=1}^{2n}\in \End(\C^{2n})\tp \C_\hbar[G]$
modulo the relations
\be
S_{12}A_2S_{12}A_2=A_2S_{12}A_2S_{12}
,\quad A_2S_{12}A_2\kappa=- q^{-{2n}-1}\kappa=\kappa A_2S_{12}A_2.
\label{A-matrix}
\ee
These relations are understood in $\End(\C^{2n})\tp \End(\C^{2n})\tp \C_\hbar[G]$,
and the indices distinguish the two copies of $\End(\C^{2n})$, in the usual
way. Note that the factor $- q^{-{2n}-1}$ before $\kappa$ is missing in $\cite{M2}$.

The algebra $\C_\hbar[G]$ is a quantization of $\C[G]$. It is the quotient of
the well known "reflection equation algebra" defined through the left identity
in (\ref{A-matrix}). From the quantization point of view, it was studied in \cite{DM1} and
\cite{M1}. Note that it is different from the $RTT$-quantization of $\C[G]$
 and is not a Hopf algebra.
The algebra $\C_\hbar[G]$ carries a $U_\hbar(\g)$-action, which is  a deformation of the coadjoint action of
$U(\g)$ on $\C[G]$. This action is set up  as
$$
(\id \tp x)(A)=(\pi(\gm(x^{(1)}))\tp \id )(A)\bigl(\pi(x^{(2)})\tp \id \bigr), \quad x\in U_\hbar(\g),
$$
on the entries of the matrix $A$.
When extended further to the free algebra generated by $\{A_{ij}\}$, this action makes it a
$U_\hbar(\g)$-module algebra. Relations  (\ref{A-matrix}) are invariant, therefore
$\C_\hbar[G]$ becomes  a $U_\hbar(\g)$-module algebra, too.

It is important that $\C_\hbar[G]$ can be realized as an invariant subalgebra in $U_q(\g)$,
where the latter is regarded as the adjoint module.
The assignment
$$
A\mapsto (\pi\tp \id)(\Ru_{21}\Ru)=\Q\in \End(\C^{2n})\tp  U_q(\g),
$$
preserves the relations (\ref{A-matrix}) and determines an embedding $\C_\hbar[G]\subset U_q(\g)$, \cite{M1}.
Relations  (\ref{A-matrix}) amount to the Yang-Baxter equation on $\Ru$ and
the existence of the one-dimensional invariant in $\C^{2n}\tp \C^{2n}$ (quantum
"symplectic form").

The following properties of $\C_\hbar[G]$  will be of importance.
Denote by $I_\hbar(G)\subset \C_\hbar[G]$ the subalgebra of $U_\hbar(\g)$-invariants.
It coincides with the centre of $\C_\hbar[G]$ and is generated by the
q-traces $\Tr_q(A^l)$, $l=1,\ldots, 2n$, which go over to $\tau_l$ under the embedding
to $U_\hbar(\g)$.
Not all traces are independent, but that is immaterial for this exposition.
The algebra $\C_\hbar[G]$ is freely generated over $I_\hbar(G)$ by
a $U_\hbar(\g)$-module whose isotypic components are finite dimensional, \cite{M1}.
This is a quantum version of the Kostant-Richardson theorem, \cite{Rich}.

Our approach to quantization is based on the following strategy that is similar to
\cite{M2}.
Suppose we have constructed  two $U_\hbar(\g)$-algebras $S_\hbar$ and  $T_\hbar$
along with an equivariant homomorphism $\varphi\colon S_\hbar \to T_\hbar$
obeying the following conditions:
1) all isotypic components in $S_\hbar$ are $\C[\![\hbar]\!]$-finite,
2) $T_\hbar$ has no $\hbar$-torsion (multiplication by $\hbar$ is injective),
3) there is an ideal $J_\hbar \subset \ker \varphi$ such that
the image $J_0^\flat$ of $J_0=J_\hbar/\hbar J_\hbar$ in $S_0=S_\hbar/\hbar S_\hbar$ is a maximal $\g$-invariant ideal in $S_0$,
4) $S_0$ is commutative.
Then a) the kernel of $\varphi$ coincides with $J_\hbar$,
b) $\varphi(S_\hbar)$ is a quantization of the algebra $S_0/J_0^\flat$.
Remark that if $S_0$ is the coordinate ring of a $\g$-variety,
maximal proper $\g$-invariant ideals are exactly the radical ideals of orbits in it.

In our situation, $T_\hbar=\End(M_\la)$ is the algebra of linear endomorphisms
of $M_\la$ and  $S_\hbar$ is the quotient of $\C_\hbar[G]$ by
the ideal generated by $\ker\chi^\la$.
Explicitly, this ideal is determined by the relations (\ref{char_V}).
By Proposition \ref{M_la is free}, $M_\la$  and hence $\End(M_\la)$ are free over $\C[\![\hbar]\!]$.
 Note that we cannot take simply $\C_\hbar[G]$
for the role of $S_\hbar$, because the isotypic components of $\C_\hbar[G]$ are not finite
due to the large centre $I_\hbar(G)$ (the subalgebra of invariants). This centre
is reduced to scalars in $S_\hbar$, which therefore has finite isotypic components by
the quantum Richardson theorem.

 The composition of the embedding $\C_\hbar[G]\to U_q(\g)$
and the representation homomorphism $U_q(\g)\to \End(M_\la)$
yields a representation of $\C_\hbar[G]$ and
factors through the homomorphism $\varphi$.
The defining ideal of a class in $G$ is a maximal $G$-invariant proper ideal
in $\C[G]$, therefore its projection to $S_0$ is
a proper  maximal $G$-invariant ideal too. Thus, to construct the quantization,
it is sufficient to check that  $\varphi$ annihilates an ideal
that turns into the defining ideal of the class in the classical limit.
As $\ker \chi^\la$ is already factored out in $S_\hbar$, we need
to check the polynomial equation on $\Q$. That has been done
in Proposition \ref{propn_evQM}.

There is an issue about the action of $U_\hbar(\g)$ as mentioned in Section \ref{secGVMM}.
The quantum group $U_\hbar(\g)$ cannot act on the $U_q(\g)$-module $M_\la$ because
the operators from $\h$ are irregular in $\hbar$ for
$\la\in \mathfrak{C}_{\k,reg}^*$. This can be fixed as follows. The image of $\C_\hbar[G]$
is contained in the subalgebra  $\End^\circ(M_\la)$ of locally finite endomorphisms of
$M_\la$.  We extend $M_\la$ by the Laurent series in $\hbar$, to enable the action of $U_\hbar(\g)$.
This action gives rise to the natural adjoint action on  $\End^\circ(M_\la)(\!(\hbar)\!)$.
It is easy to see that the $\C[\![\hbar]\!]$-submodule $\End^\circ(M_\la)\subset \End^\circ(M_\la)(\!(\hbar)\!)$
is $U_{\hbar}(\g)$-invariant. Indeed, $\End^\circ(M_\la)$ is a weight module and all its weights
belong to $\Z\Pi^+$. Therefore, the action of $\h$ on $\End^\circ(M_\la)$ is correctly defined.
The adjoint action of the  Chevalley generators on $\End^\circ(M_\la)$  brings about operators $q^{\h}\subset \End^\circ(M_\la)$, through
the comultiplication and antipode. This reasoning proves that,
albeit $U_\hbar(\g)$ is not represented on $M_\la$, the "adjoint" action on  $\End^\circ(M_\la)$ is well defined.
\begin{thm}
\label{QCC}
Suppose that $\la=\mathfrak{C}_{\k,reg}^*$ and let $\mub$ be as defined in (\ref{mu_param}).
The quotient of $\C_\hbar[G]$ by the ideal of relations
\be
\prod_{i=1}^{\ell}(\Q-\mu_i)\times(\Q-\mu_{\ell+1})(\Q-\mu_{\ell+2})\times
\prod_{i=1}^{\ell}(\Q-\mu_i^{-1}q^{-4n+2(n_i-1)})=0,
\label{q-min_pol}
\ee
\be
\Tr_q(\Q^k)=\vt_{\nb,q}^k(\mub)
\label{q-traces}
\ee
is an equivariant quantization of the class $\mub^0=\hat \Mc_K$,
where $\mub^0=\lim_{\hbar\to 0}\mub$. It is the image of $\C_\hbar[G]$ in the algebra of endomorphisms
of the $U_q(\g)$-module $M_\la$.
\end{thm}
Theorem \ref{QCC} describes quantization in terms of the matrix $\Q$, which is the image of the matrix $A$.
To obtain the description in terms of $A$, one should replace $\Q$ with $A$ and add the relations
(\ref{A-matrix}).

The constructed quantization is equivariant with respect to the standard or Drinfeld-Jimbo quantum group
$U_\hbar(\g)$. Other quantum groups are obtained from standard by twist, \cite{ESS}.
Formulas (\ref{q-min_pol}) and (\ref{q-traces}) are valid for
any quantum group  $U_\hbar(\g)$ upon the following modifications. The matrix $\Q$ is expressed through the universal R-matrix as
usual. The q-traces should be redefined
as $\tau_k=\Tr_q(\Q^k)=q^{1+2n}\Tr\Bigl(\pi\bigl(\gm^{-1}(\Ru_1)\Ru_2\bigr)\Q^k\Bigr)$,
where $\gm$ is the antipode and $\Ru$ is the universal R-matrix of $U_\hbar(\g)$.
This can be verified along the lines of \cite{MO}.

\section{Appendix}
Below we present the diagrams $D_1,D_2,D_3$ in order to illustrate the formalism of Section \ref{secWMsym} for the case of $m=1$, $p=2$, $n=3$.
$$
\begin{array}{ccccccccccc|c}
&&&&&D_1\\
&\scriptstyle{f_{\al_1}}&&\scriptstyle{f_{\al_2}}&&\scriptstyle{f_{\bt}}&&\scriptstyle{f_{\al_2}}&&\scriptstyle{f_{\al_1}}&\\\hline
\scriptstyle{w_1\tp \phi_1^5 v}&\leftarrow& \scriptstyle{w_1\tp \phi_1^4 }v&\leftarrow& \scriptstyle{w_1\tp \phi_1^3 v}
&\leftarrow&\scriptstyle{w_1\tp \phi_1^2 v} &\leftarrow&\scriptstyle{w_1\tp \phi_1^1 v}&\leftarrow&\scriptstyle{w_1\tp v}
\\
&&\downarrow&&\downarrow&&\downarrow&&\downarrow&&\quad\downarrow &\scriptstyle{f_{\al_1}}\\
&&\scriptstyle{w_2\tp \phi_1^4 v}&\leftarrow&\scriptstyle{w_2\tp \phi_1^3 v}&\leftarrow&\scriptstyle{w_2\tp \phi_1^2 v}&\leftarrow&\scriptstyle{w_2\tp \phi_1^1 v}&\leftarrow&\scriptstyle{w_2\tp v}\\
&&&&\downarrow&&\downarrow&&\downarrow&&\quad\downarrow &\scriptstyle{f_{\al_2}}\\
&&&&\scriptstyle{w_3\tp \phi_1^3 v}&\leftarrow&\scriptstyle{w_3\tp \phi_1^2 v}&\leftarrow&\scriptstyle{w_3\tp \phi_1^1 v}&\leftarrow&\scriptstyle{w_3\tp v}\\
&&&&&&\downarrow&&\downarrow&&\quad\downarrow &\scriptstyle{f_{\bt}}\\
&&&&&&\scriptstyle{w_4\tp \phi_1^2 v}&\leftarrow&\scriptstyle{w_4\tp \phi_1^1 v}&\leftarrow&\scriptstyle{w_4\tp v}\\
&&&&&&&&\downarrow&&\quad\downarrow &\scriptstyle{f_{\al_2}}\\
&&&&&&&&\scriptstyle{w_5\tp \phi_1^1}&\leftarrow&\scriptstyle{w_5\tp v}\\
&&&&&&&&&&\quad\downarrow &\scriptstyle{f_{\al_1}}\\
&&&&&&&&&&\scriptstyle{w_6\tp v}\\
\end{array}
$$
$$
\begin{array}{ccccccccccc|c}
&&&&&D_2\\
&\scriptstyle{f_{\al_2}}&&\scriptstyle{f_{\bt}}&&\scriptstyle{f_{\al_1}}&&\scriptstyle{f_{\al_2}}&&\scriptstyle{f_{\al_1}}&\\\hline
\scriptstyle{w_1\tp \phi_2^5 v}&\leftarrow&\scriptstyle{w_1\tp \phi_2^4 v}&\leftarrow&\scriptstyle{w_1\tp \phi_2^3 v}&\leftarrow&\scriptstyle{w_1\tp \phi_2^2 v}&\leftarrow&\scriptstyle{w_1\tp \phi_2^1 v }&\leftarrow&\scriptstyle{w_1\tp v}\\
&&\downarrow&&\downarrow&&\downarrow&&\downarrow&&\downarrow &\scriptstyle{f_{\al_1}}\\
&&\scriptstyle{w_2\tp \phi_2^4 v}&\leftarrow&\scriptstyle{w_2\tp \phi_2^3 v}&\leftarrow&\scriptstyle{w_2\tp \phi_2^2 v}&\leftarrow&\scriptstyle{w_2\tp \phi_2^1 v}&\leftarrow&\scriptstyle{w_2\tp v}\\
&&&&\downarrow&&\downarrow&&\downarrow&&\downarrow &\scriptstyle{f_{\al_2}}\\
&&&&\scriptstyle{w_3\tp \phi_2^3 v}&\leftarrow&\scriptstyle{w_3\tp\phi_2^2 v}&\leftarrow&\scriptstyle{w_3\tp\phi_2^1 v}&\leftarrow&\scriptstyle{w_3\tp v}\\
\end{array}
$$
$$
\begin{array}{ccccccccccc|c}
&&&&&D_3\\
&\scriptstyle{f_{\bt}}&&\scriptstyle{f_{\al_2}}&&\scriptstyle{f_{\al_1}}&&\scriptstyle{f_{\al_2}}&&\scriptstyle{f_{\al_1}}&\\\hline
\scriptstyle{w_1\tp \phi_3^5 v}&\leftarrow&\scriptstyle{w_1\tp \phi_3^4 v}&\leftarrow&\scriptstyle{w_1\tp \phi_3^3 v}&\leftarrow&
\scriptstyle{w_1\tp \phi_3^2 v}&\leftarrow&\scriptstyle{w_1\tp \phi_3^1 v}&\leftarrow&\scriptstyle{w_1\tp v}\\
&&\downarrow&&\downarrow&&\downarrow&&\downarrow&&\downarrow &\scriptstyle{f_{\al_1}}\\
&&\scriptstyle{w_2\tp \phi_3^4 v}&\leftarrow&\scriptstyle{w_2\tp \phi_3^3 v}&\leftarrow&\scriptstyle{w_2\tp \phi_3^2 v}&\leftarrow&\scriptstyle{w_2\tp \phi_3^1 v}&\leftarrow&\scriptstyle{w_2\tp v}\\
&&&&\downarrow&&\downarrow&&\downarrow&&\downarrow &\scriptstyle{f_{\al_2}}\\
&&&&\scriptstyle{w_3\tp \phi_3^3 v}&\leftarrow&\scriptstyle{w_3\tp \phi_3^2 v}&\leftarrow&\scriptstyle{w_3\tp \phi_3^1v}&\leftarrow&\scriptstyle{w_3\tp v}\\
\end{array}
$$

\vspace{20pt}
{\bf Acknowledgements}.
This research is supported in part by the RFBR grant 12-01-00207-a.
The author is extremely grateful to Joseph Bernstein, Steve Donkin, and Liam O'Carroll for illuminating discussion
of Proposition \ref{radical_gen}. The final version of the paper has benefited form valuable remarks and suggestions
of the referees, to whom we are much indebted.
Our special thanks are to the Max-Plank Institute for Mathematics in Bonn for hospitality.

\end{document}